\documentclass[a4paper,11pt]{article}

\title{Indirect Inference for Time Series Using the Empirical Characteristic Function and Control Variates}
  
\author{
Richard A. Davis
\thanks{Department of Statistics, Columbia University, 1255 Amsterdam Avenue, New York, NY 10027, USA, email: rdavis@stat.columbia.edu}
\and Thiago do R\^ego Sousa 
\thanks{Center for Mathematical Sciences, Technical University of Munich,  85748 Garching, Boltzmannstr.~3, Germany, email:  thiago.sousa@tum.de, cklu@tum.de}
\and Claudia Kl\"uppelberg\footnotemark[2]
}
       \usepackage{graphicx}
       \usepackage{braket} 
       \usepackage{amsmath} 
       \usepackage{bm} 
       \usepackage{amsfonts,amsthm,amstext,amssymb,amsmath}
       \usepackage{float}
       \usepackage[toc,page]{appendix}
       \usepackage{pdfpages}
       \usepackage{caption}
       \usepackage{color}
       \usepackage{comment}
       \usepackage[round]{natbib} 
       \usepackage{blkarray} 
	   \usepackage{multirow} 
	   \usepackage{dsfont}
	   \usepackage{enumitem}
	   \usepackage{kantlipsum} 
	   \usepackage[linktocpage=true]{hyperref}
     \usepackage{graphicx}
       
\numberwithin{equation}{section}

\newtheorem{theorem}{Theorem}[section]
\newtheorem{lemma}[theorem]{Lemma}
\newtheorem{remark}[theorem]{Remark}
\newtheorem{example}[theorem]{Example}
\newtheorem{proposition}[theorem]{Proposition}
\newtheorem{definition}[theorem]{Definition}
\newtheorem*{assumptionA}{Assumptions A}
\newtheorem*{assumptionB}{Assumptions B}
\newtheorem*{assumptionC}{Assumptions C}
\newtheorem*{assumptionD}{Assumptions D}

\newtheorem{corollary}[theorem]{Corollary}

\newcommand{\bthe}{\begin{theorem}}
\newcommand{\ethe}{\end{theorem}}

\newcommand{\ble}{\begin{lemma}}
\newcommand{\ele}{\end{lemma}}

\newcommand{\bde}{\begin{definition}\rm}
\newcommand{\ede}{\Chalmos\end{definition}}

\newcommand{\bco}{\begin{corollary}}
\newcommand{\eco}{\end{corollary}}

\newcommand{\bpr}{\begin{proposition}}
\newcommand{\epr}{\end{proposition}}

\newcommand{\brem}{\begin{remark}\rm}
\newcommand{\erem}{\Chalmos\end{remark}}

\newcommand{\bproof}{\begin{proof}}
\newcommand{\eproof}{\end{proof}}

\newcommand{\bexam}{\begin{example}\rm}
\newcommand{\eexam}{\Chalmos\end{example}}

\newcommand{\beao}{\begin{eqnarray*}}
\newcommand{\eeao}{\end{eqnarray*}\noindent}

\newcommand{\beam}{\begin{eqnarray}}
\newcommand{\eeam}{\end{eqnarray}\noindent}

\newcommand{\barr}{\begin{array}}
\newcommand{\earr}{\end{array}}

\newlist{Aenumerate}{enumerate}{1}
\setlist[Aenumerate]{label=A.\arabic*}
\definecolor{lg}{gray}{0.9} 
\def\N{{\mathbb N}}
\def\C{{\mathbb C}}
\def\Z{{\mathbb Z}}

\def\E{{\mathbb E}}
\def\R{{\mathbb R}}

 \def\1{\mathds{1}}

\newcommand{\stp}{\stackrel{P}{\rightarrow}}
\newcommand{\std}{\stackrel{d}{\rightarrow}}
\newcommand{\stas}{\stackrel{\rm a.s.}{\rightarrow}}
\newcommand{\simiid}{\stackrel{\rm iid}{\sim}}

\newcommand{\eqd}{\stackrel{\mathrm{d}}{=}}
\newcommand{\eqas}{\stackrel{\mathrm{a.s.}}{=}}

\newcommand{\nto}{{n\to\infty}}
\newcommand{\deltato}{{\delta\to\infty}}

\newcommand{\vp}{\varphi}
\newcommand{\eps}{\varepsilon}

\newcommand{\la}{\langle}
\newcommand{\ra}{\rangle}

\DeclareMathOperator*{\argminA}{arg\,min} 
\newcommand{\var}{\mathbb{V}{\rm ar}}
\newcommand{\cov}{\mathbb{C}{\rm ov}}

\newcommand{\bsX}{{\bm{X}}}

\newcommand{\tildeX}{\tilde{X}}
\newcommand{\bstildeX}{\tilde{\bm{X}}}

\newcommand{\diff}{{\rm d}}

\newcommand{\gradtheta}{\nabla_{\theta}}

\newcommand{\intrp}{\int_{\R^p}}
\newcommand{\suptheta}{\sup_{\theta \in \Theta} }
\newcommand{\thehat}{\hat{\theta}_{n,H}}
\newcommand{\thehatcv}{\hat{\theta}_{n,H,k}^{\text{(cv)}}}

\newcommand{\kdelta}{K_{\delta}}
\newcommand{\diffthetai}{\frac{\partial}{\partial{\theta^{(i)}}}}
\newcommand{\diffthetak}{\frac{\partial}{\partial{\theta^{(k)}}}}
 \newcommand{\diffthetaki}{\frac{\partial}{\partial{\theta^{(k)}}\partial{\theta^{(i)}}}}

\newcommand{\ov}{\overline}
\newcommand{\wh}{\widehat}

\newcommand{\Chalmos}{\quad\hfill\mbox{$\Box$}}  

\newcommand{\ts}[1]{{\color{magenta} #1}}

\allowdisplaybreaks[3]

\begin{document}
\maketitle

\begin{abstract}
We estimate the parameter of a stationary time series process by minimizing the integrated weighted mean squared error between the empirical and simulated characteristic function, when the true characteristic functions cannot be explicitly computed. 
Motivated by Indirect Inference, we use a Monte Carlo approximation of the characteristic function based on iid simulated blocks. 
As a classical variance reduction technique, we propose the use of control variates for reducing the variance of this Monte Carlo approximation. 
These two approximations yield two new estimators
that are applicable to a large class of time series processes. 
We show consistency and asymptotic normality of the parameter estimators under strong mixing, moment conditions, and smoothness of the simulated blocks with respect to its parameter. 
In a simulation study we show the good performance of these new simulation based estimators, and the superiority of the control variates based estimator for Poisson driven time series of counts. 

\end{abstract}


%
%

\section{Introduction}

Let $(X_j)_{j \in \Z}$ be a stationary time series, whose distribution depends on $\theta \in \Theta \subset \R^q$ for some $q \in \N$. 
Denote by $\theta_0\in \Theta$ the true parameter, which we want to estimate from observations $X_1,\dots,X_{T}$ of the time series.
Maximum likelihood estimation (MLE) has been extensively used for parameter estimation, since under weak regularity conditions it is known to be asymptotically efficient. {For many models, however, MLE is not always feasible to carry out, due to a likelihood that may be intractable to compute, or maximization of the likelihood is difficult, {or because the likelihood function is unbounded on $\Theta$.}} To overcome such problems, alternative methods have been developed, for instance,  the generalized method of moments (GMM)  in \citet{Hansen82}, the quasi-maximum likelihood estimation (QMLE) in \citet{White82}, and composite likelihood methods in \citet{lindsay1988composite}.

In a similar vein, \cite{Feuerverger1990} 
proposed an estimator based on matching the empirical characteristic function (chf) computed from blocks of the observed time series and the true chf. 
More specifically, given a fixed $p \in \N$, the observed blocks of $X_1,\dots,X_{T}$ are
\begin{equation}\label{eq:blockXj}
\bsX_j = (X_j,\dots,X_{j+p-1}), \quad j = 1,\dots,n,
\end{equation}
where $n = T-p+1$. 
In that paper, a finite set of points in $\R^p$ needs to be chosen as arguments for {which} the true and the empirical chf {are compared}. 
However, the practical choice of this set depends on the problem at hand and the asymptotic results derived in \citet{Feuerverger1990} do not offer practical guidance {for choosing these points}.
 To overcome this {limitation} \cite{yu1998empirical} and \citet{Knight02Corr} {considered a integrated weighted squared distance} 
between the empirical and the true {chfs}.
 

This method has been used in a variety  of applications; an interesting review paper, \cite{Yu} contains a wealth of examples and references. More recent publications, where the method has been successfully applied to discrete-time models include  \citet*{knight2002SVmodelECF}, \citet{meintanis2012inference}, \citet{kotchoni2012applications}, \citet*{milovanovic2014application}, \citet{francq2016fourier}, and \citet{ndongo2016estimation}. 
The method also applies to continuous-time processes after discretization and has been used prominently for L\'evy-driven models. The book \cite{LM4} {provides additional} insight and references in this field.

The principal goal of this paper is to extend the ideas of these papers to a more general setting. For example, we do not assume the idealized situation for which the chf has an explicit expression as a function of $\theta \in \Theta$. We propose two new estimators of $\theta$, which are based on replacing the true chf with estimates that are constructed from a functional approximation of the chf constructed from simulated sample paths of $(X_j(\theta))_{j \in \Z}$.

While much attention has been given to the choice of the {integrated} distance used when computing such estimators, which under some regularity conditions can achieve the Cram\'{e}r-Rao efficiency bound (see eq. (2.3) of \citet{Knight02Corr} and Proposition~4.2 of \citet{Carrasco07}), the focus of our paper is on the practical and theoretical aspects that emerge when it is required to approximate the theoretical chf for parameter estimation. For more details on the search for efficient estimators we refer to \cite{Carrasco07, carrasco2014asymptotic, carrasco2017efficient}.

Our first estimator is computed from a simple Monte Carlo approximation to replace the true, but unknown chf.
This is similar to the simulated method of moments of \citet{Mcfadden89} and of the indirect inference method (\cite{Smith93II} and \citet{Gourieroux93}). 
In particular, indirect inference has been successfully applied in a variety of situations: parameter estimation of continuous time models with stochastic volatility (\citet{Bianchi96}, \citet{Jiang98}, \citet{Raknerud12}, \citet{Laurini13} and \citet*{Wahlberg15}), robust estimation (\citet{deLuna01} and \citet{Fasen18II}), and finite sample bias reduction (\citet*{Gourieroux00,Gourieroux10} and \citet*{Thiago1}).

More precisely, for many different $\theta \in \Theta$, we simulate an iid sample of blocks denoted by
\begin{equation}\label{eq:XjhblcksINTRO}
\bstildeX_j(\theta) = (\tildeX_1^{(j)}(\theta),\dots,\tildeX_p^{(j)}(\theta)), \quad j = 1,\dots,H,
\end{equation}
for $H \in \N$, and define a {\em simulation based parameter estimator}, which minimizes the integrated weighted mean squared error, which is the integrated distance we use, between the empirical chf computed from the blocks \eqref{eq:XjhblcksINTRO} of the observed time series and its simulated version computed from a large number of simulated paths of the time series. 

This is in contrast to the simulation based estimator defined in Section~5.2 of \citet{Carrasco07}, which is computed from one long time series path instead of the iid sample of blocks in \eqref{eq:XjhblcksINTRO} (a similar method has been applied by \citet{Forneron18} to estimate the structural parameters and the distribution of shocks in dynamic models).
Since we compute the Monte Carlo approximation of the chf from independent blocks, it should have smaller variance than the corresponding one for dependent blocks.  
{Our method} gives a chf approximation which yields strongly consistent and asymptotically normal parameter estimators.
We also report their small sample properties for different models.


Furthermore, as the Monte Carlo approximation of the chf is computed from iid blocks of a time series, control variates techniques (see \cite{glynn2002CV} and \cite{casella2013monte}) provide an even more accurate approximation for the chf. 
Control variates techniques are classical variance reduction methods in simulation. 
The {idea is to use} a set of control variates, which are correlated with the chf. The method then approximates the joint covariance matrix of the control variates and the chf, and uses it to construct a new Monte Carlo approximation of the chf.
{We choose the first two terms in the {Taylor} expansion of the complex exponential $e^{i \la t,\bsX_1(\theta)\ra}$}, $\la t,\bsX_1(\theta)\ra$ and $\la t,\bsX_1(\theta)\ra^2$ for $\theta \in \Theta$ as control variates, where $\la \cdot, \cdot \ra$ denotes the usual Euclidean {inner} product in $\R^d$. {This requires knowing the mean and covariance matrix of $\bsX_1(\theta)$ for $\theta \in \Theta$}.

{In assessing} the performance of both the Monte Carlo approximation and the control variates approximation of the chf, {two trends emerge}. 
First, both the Monte Carlo and the control variates approximations work better for small values of the argument. 
Second, the control variates approximation performs much better than the Monte Carlo approximation, in particular, for small values of the argument.
As a consequence, we propose a {\em control variates based parameter estimator} whose {integrated mean squared error} distinguishes between small and large values of the argument. 

Under regularity conditions we prove strong consistency of {the proposed} parameter estimators and asymptotic normality of the simulation based parameter estimator. We find that the simulation based parameter estimator is asymptotically normal with asymptotic covariance matrix equal to the one of the oracle estimator as derived in \cite{Knight02Corr}. From this we conclude that there cannot be any improvement in the limit law for the asymptotic normality of the control variates based estimator. However, we prove that it is computed from a better approximation of the chf. Thus, the control variates estimator improves the finite sample performance compared to the simulation based parameter estimator.

It is assumed throughout that $(X_j)_{j \in \mathbb{Z}}$ is a stationary time series. This ensures that the blocks of random variables in (1.1) are stationary, from which we obtain convergence of the empirical chf to the joint chf. Now in some restricted cases, our method can be adapted to special types of nonstationarity. For example, if $(X_j)_{j \in \mathbb{Z}}$ is nonstationary, but the differenced process $\nabla X_j = X_j - X_{j-1}$ is stationary, then our methodology can be applied directly to $\nabla X_j$. Similarly, if $X_j = Y_j + \mu_j$, where $Y_j$ is stationary and $\mu_j$ is a mean function that can be estimated consistently say by $\hat{\mu}_j$, then the methodology can be applied to $X_j - \hat{\mu}_j$. We do not pursue this line of investigation here.

The finite sample performance of the estimators are investigated for two important models. 
We begin with a stationary Gaussian ARFIMA model, whose chf is explicitly known so that we can use the oracle estimator and compare its performance with the simulated based estimator. 
Their performance is comparable and also very close to the MLE, so in this model there is no need to use control variates. {The} second example is a nonlinear model for time series of counts, which has been proposed originally in \citet{Zeger88} and applied, for instance, for modeling disease counts (see also \citet{campbell1994time}, \citet{chan1995monte} and \citet*{davis1999modeling}).  

{In the second example,} the oracle estimator does not apply, since the chf of a Poisson-AR process cannot be computed in closed form. 
For this model and different parameter sets, both the simulation based and the control variates based estimators perform satisfactorily, and the control variates based estimator improves the performance of the simulation based estimator considerably. 
When compared with the composite pairwise likelihood estimator in \citet{Davis11PL}, the control variates based estimator has comparable or even smaller bias. 

Our paper is organized as follows. In Section~\ref{s2} we present the oracle estimator, and the estimators computed from a Monte Carlo approximation and from a control variates approximation of the chf in detail. Here we also motivate the choice of the control variates used. 
The asymptotic properties of the two new estimators are {established} in Section~\ref{se:asymp}. 
As all estimators are computed from true or approximated chf's we assess their {performance} in Section~\ref{se:mc_vs_cv}, first for a Gaussian AR(1) process and then for the Poisson-AR process. Practical aspects of calculating the weighted least squares function are discussed in 
Section~\ref{s5}, as well as the estimation results for finite samples. In Section~\ref{s52} we compare the oracle estimator, the simulation based parameter estimator and the MLE for a Gaussian ARFIMA model, whereas in Section~\ref{s51} we compare the simulation based parameter estimator and the control variates based estimator for the Poisson-AR process. 
The proofs of the main results in Section~\ref{se:asymp}, of Lemma~1 of Section~\ref{s5}, and the Tables discussed in Sections~\ref{s52} and ~\ref{s51} are provided in the Appendix.

\section{Parameter estimation based on the empirical characteristic function}\label{s2}

Throughout we use the following notation. 
For $z\in\C$ we use the $L^2$-norm: $|z|=\sqrt{z\,\ov z}$, where $\ov z$ is the complex conjugate of $z$.
For $x\in\R^d$ and $d\in\N$ we denote by $|x|$ the $L^2$-norm, but recall that in $\R^d$ all norms are equivalent.
For $z \in \C$ the symbols $\Re(z)$ and $\Im(z)$ denote its real and imaginary part. 
For a function $f: \R^q\to \R^p$ its Jacobi matrix is given by $\gradtheta f(\theta) = \frac{\partial f(\theta)}{\partial \theta^T} \in \R^{p \times q}$ and $\gradtheta^2 f(\theta) = \frac{\partial \text{vec}(\gradtheta f(\theta)) }{\partial \theta^T} \in \R^{pq \times q}$.

\subsection{The oracle estimator}\label{s21}

Let $(X_j(\theta))_{j \in \Z}$ be a stationary time series process, whose distribution depends on $\theta \in \Theta \subset \R^q$ for some $q \in \N$. 
Denote by $\theta_0\in \Theta$ the true parameter, which we want to estimate, and suppose that we observe $X_1,\dots,X_{T}$. 
Given a fixed $p \in \N$, define for $\theta\in\Theta$ the $p$-dimensional blocks
\begin{equation}\label{eq:blocksXj}
\bsX_j (\theta)=  (X_j(\theta),\dots,X_{j+p-1}(\theta)), \quad j = 1,\dots,n,
\end{equation}
 where $n = T-p+1$. {For $j = 1,\dots,n,$ the observed blocks correspond to $\bsX_j = (X_j,\dots,X_{j+p-1})$}, which can be used to calculate the {\em empirical characteristic function (chf)},  defined as 
\begin{equation}\label{eq:dfepcf}
 \varphi_n(t) =  \frac{1}{n} \sum_{j=1}^n e^{i \la t,\bsX_j \ra }, \quad t \in \R^p.
\end{equation}
Under mild conditions such as ergodicity, $\varphi_n(t)$ converges a.s. pointwise to the true chf $\vp(t) = \E e^{i\la t ,\bsX_1 \ra}$ for all $t \in \R^p$. We assume that $p$ is chosen in such a way that $\vp(\cdot)$ uniquely identifies the parameter of interest $\theta$. The idea of estimating $\theta_0$ from a single time series observation by matching the empirical chf of blocks of the observed time series and the true one has been proposed in \cite{yu1998empirical} and \citet{Knight02Corr}, and we use the one in \cite{Knight02Corr}, where the {\em oracle estimator} of $\theta_0$ is defined as
\begin{equation}\label{eq:defOra}
\hat{\theta}_n = \text{argmin}_{\theta \in \Theta} Q_n(\theta),
\end{equation}
where
\begin{equation}\label{eq:defQast}
 Q_n(\theta) = \int_{\mathbb{R}^p} |  \vp_n(t) -  \vp(t,\theta)|^2 w(t) \diff t, \quad \theta \in \Theta,
\end{equation}
with suitable  weight function $w$ such that the integral is well-defined, and chf
\begin{equation}\label{eq:cf}
\vp(t,\theta) = \E e^{i\la t, \bsX_1(\theta) \ra}, \quad t \in \R^p.
\end{equation}

In an ideal situation, $\vp(\cdot,\theta)$ has an explicit expression, which is known for all $\theta\in\Theta$.

\subsection{Estimator based on a Monte Carlo approximation of \texorpdfstring{$\vp(\cdot,\theta)$}{Lg}}\label{s22}

Unfortunately, a closed form expression of the chf $\vp(\cdot,\theta)$ is for many time series processes not available. 
However, it can be approximated by a Monte Carlo simulation, and an idea borrowed from the simulated method of moments (\citet{Mcfadden89}, see also \cite{Smith93II} and \citet*{Gourieroux93} for a similar idea in the context of indirect inference) is to replace $\vp(\cdot,\theta)$ by its functional approximation constructed from simulated sample paths of $(X_j(\theta))_{j \in \Z}$. 
For many different $\theta \in \Theta$, we simulate, {independent of {the observed time series}, 
an iid sample of the blocks in \eqref{eq:blocksXj} denoted by
\begin{equation}\label{eq:Xjhblcks}
\bstildeX_j(\theta) = (\tildeX_1^{(j)}(\theta),\dots,\tildeX_p^{(j)}(\theta)), \quad j = 1,\dots,H,
\end{equation}
for $H \in \N$, and define the {\em Monte Carlo approximation} of $\vp(\cdot,\theta)$ based on these simulations as
\begin{equation}\label{eq:ecfsim}
 \vp_H(t,\theta) =  \frac{1}{H} \sum_{j=1}^H e^{i \la t,\bstildeX_j(\theta) \ra } , \quad t \in \R^p.
\end{equation}
If we replace $\vp(\cdot,\theta)$ in \eqref{eq:defQast} by $\vp_H(\cdot,\theta)$, we obtain the {\em simulation based parameter estimator}
\begin{equation}\label{eq:defThe}
\thehat =  \argminA_{\theta \in \Theta} Q_{n,H}(\theta),
\end{equation}
where
\begin{equation}\label{eq:defQ}
 Q_{n,H}(\theta) = \int_{\mathbb{R}^p} |  \vp_n(t) -   \vp_H(t,\theta)|^2 w(t) \diff t,
\end{equation}
with suitable  weight function $w$ such that the integral is well-defined.

\begin{remark}
An alternative {approximation} to \eqref{eq:ecfsim} of the chf is based on generating one long time series path and use the empirical chf of the consecutive blocks of $p$-dimensional random variables constructed as in \eqref{eq:blocksXj} (see \citet{Carrasco07}). While being unbiased, the {approximation} will generally have larger variance than the {approximation}  \eqref{eq:ecfsim}. {Nevertheless,  when it is expensive to generate realizations even of dimension $p$, for instance, when a long burn-in time is required to achieve stationarity, it may be computationally more efficient to generate one long time series. While we do not pursue this approach here, the technical aspects of working with one long time series are not much different than the estimate based on independent replicates as in \eqref{eq:ecfsim}, but might require a much larger sample
size than desired to control the variance of the estimate.  This is especially true for long-memory time series.}
\end{remark}

Since $\vp_{H}(\cdot,\theta)$ is based on $H$ iid time series blocks, we can reduce its variance further using control variates to produce an even more accurate approximation for the chf.
This will result in an improved version of $\hat{\theta}_{n,H}$.

\subsection{Estimator based on a control variates approximation of \texorpdfstring{$\vp(\cdot,\theta)$}{Lg}}

The estimator $\hat{\theta}_{n,H}$ in \eqref{eq:defThe} requires only that the stationary time series process can be simulated, and is therefore easily applicable to a large class of models. 
When computing $Q_{n,H}(\theta)$ of \eqref{eq:defQ}, it is very important that the error
\begin{equation}\label{eq:errort}
\xi_{H}(t,\theta) = |\vp_{H}(t,\theta) - \vp(t,\theta)|, \quad t \in \R^p, \theta \in \Theta,
\end{equation}
in approximating the true chf is small, since it propagates to $\hat{\theta}_{n,H}$. In order to reduce the variance of the empirical chf $\vp_{H}(\cdot,\theta)$, we use the method of control variates, {\ts{an} often used} variance reduction technique in the context of Monte Carlo integration (\cite{glynn2002CV}, \citet*{oates2017CVgrowing}, \citet{Segers18CV}).

We construct a control variates approximation of $\vp(\cdot,\theta)$ from the iid sample $\bstildeX_j(\theta)$, $j = 1,\dots,H$, as in \eqref{eq:Xjhblcks}. We also require explicit expressions for the moments $\E \la t,\bsX_1(\theta)\ra^\nu$ for $\nu = 1,2$ and $\theta \in \Theta$.

 Recall that $\bstildeX_1(\theta) \eqd \bsX_1(\theta)$ for all $\theta\in\Theta$, so that both random variables have the same moments. 
 As in \citet{Segers18CV}, we denote by $P_\theta$ the distribution of the block $\bsX_1(\theta)$ and by $P_{H,\theta}$ its empirical version. For example, if $f_t(x) = e^{i \la t,x \ra }$ for $t,x \in \R^p$, we want to provide a good approximation for {$\vp(t,\theta) = \E f_t(\bsX_1(\theta)) =: P_\theta(f_t)$ for $\theta \in \Theta$.}
To apply the control variates technique, we need control functions, which are correlated with $f_t(\bsX_1(\theta))$ and whose expectations are known. 
In the time series context, it is often that we know the first and second order structure of the process in closed form. Even for complicated models, e.g., models defined in terms of stochastic integrals (see e.g. \cite{brockwell2001levy,Kluppelberg04,brockwell2006continuous,Stelzer10}) these expressions are available. The first and second order of $\bsX_1(\theta)$ appear in the Taylor series of $f_t(\bsX_1(\theta))$ and therefore they are natural choices of control functions. We also remark that if the time series process also allows for the computation of additional moments expressions in closed form, which are correlated with $f_t(\bsX_1(\theta))$, then we encourage using them as control functions while approximating the chf. We describe now the construction of the control variates approximation in detail.

We use the first two terms in the Taylor series of the complex function $f_t(x)$, which suggests the vector of control functions $h_{t,\theta} = (h_{1,t,\theta},h_{2,t,\theta})^T$, where for $\nu = 1,2$,
\begin{equation*}
h_{\nu,t,\theta}(x) = \la t,x \ra^\nu - \E \la t,\bsX_1(\theta) \ra ^\nu, \quad t \in \R^p,
\end{equation*}
so that $P_\theta(h_{t,\theta}) = 0$, the zero vector in $\R^2$. 
The Monte Carlo approximation of $\vp(\cdot,\theta)$ based on the iid sample $\bstildeX_j(\theta)$, $j = 1,\dots,H$, is then 
\beam\label{eq:chfH}
P_{H,\theta}(f_t) = \frac1H \sum_{j=1}^H f_t(\bstildeX_j(\theta)) =  \frac1H \sum_{j=1}^H e^{i \la t, \bstildeX_j(\theta)\ra } =\vp_H(t,\theta).
\eeam
Since $\E P_{H,\theta}(f_t) = \E f_t(\bsX_1(\theta))$, the Monte Carlo approximation $\vp_H(t,\theta)$ is unbiased and has variance 
\begin{equation}\label{eq:defst2}
\var [ P_{H,\theta}(f_t)] = H^{-1}\sigma_\theta^2(f_t)\quad\mbox{with}\quad \sigma_{\theta}^2(f_t) =P_\theta (\{f_t - P_\theta(f_t)\}^2).
\end{equation}
Then for every vector $\beta \in \C^2$, we have that $P_{H,\theta}(f_t) - \beta^T P_{H,\theta}(h_{t,\theta})$ is also an unbiased estimator of $\vp(t,\theta)$. Since $\bstildeX_j(\theta)$, $j = 1,\dots,H$, is an independent sample, {$\var[P_{H,\theta}(f_t) - \beta^T P_{H,\theta}(h_{t,\theta})] = H^{-1}\sigma_{\theta}^2(f_t - \beta^T h_{t,\theta})$}
and, if we differentiate the map $\beta \mapsto \sigma_{\theta}^2(f_t - \beta^T h_{t,\theta})$ with respect to $\beta$ and set it equal to zero, we obtain (cf. Approach~1 in \citet{glynn2002CV}) the theoretical optimum
\begin{equation}\label{eq:dfbeopt}
\beta^{(\text{opt})}_{\theta,f_{t}}(h_{t,\theta}) = \{P_\theta(h_{t,\theta}h_{t,\theta}^T)\}^{-1} P_\theta(h_{t,\theta}f_t),
\end{equation}
provided the inverse exists.
In this case, the estimator
\begin{equation}\label{eq:CVoracle}
\vp^{\text{(cvopt)}}_H(t,\theta) = P_{H,\theta}(f_t) - (\beta^{(\text{opt})}_{\theta,f_{t}}(h_{t,\theta}))^T P_{H,\theta}(h_{t,\theta})
\end{equation}
has minimal asymptotic variance.
In order to investigate the existence of the above inverse note that
for each fixed $t \in \R^p$ and $\theta \in \Theta$, {the determinant of $P_\theta(h_{t,\theta}h_{t,\theta}^T)$ is}
\begin{equation*}
\var[\la t,\bstildeX_1(\theta) \ra] \var[\la t,\bstildeX_1(\theta)\ra^2] - \{\cov[\la t,\bstildeX_1(\theta) \ra, \la t,\bstildeX_1(\theta)\ra^2]\}^2.
\end{equation*}
Since by the Cauchy-Schwarz inequality,
$$
\{\cov[\la t,\bstildeX_1(\theta) \ra, \la t,\bstildeX_1(\theta)\ra^2]\}^2 \leq \var[\la t,\bstildeX_1(\theta) \ra]\var[\la t,\bstildeX_1(\theta)\ra^2],
$$
it follows (see e.g. \citet{Klenke13Prob}, Theorem~5.8) that
\begin{equation}\label{eq:abcas0}
    \det(P_\theta(h_{t,\theta}h_{t,\theta}^T) = 0 \iff a \la t,\bstildeX_1(\theta)\ra + b \la t,\bstildeX_1(\theta)\ra^2 + c \eqas 0,
\end{equation}
for some $a,b,c \in \R$ with $|a|+|b|+|c| > 0$. 
As the scalar product is random, universal coefficients to satisfy the right-hand side of \eqref{eq:abcas0} exist only in degenerate cases, which we do not consider.

Since $\beta^{(\text{opt})}_{\theta,f_{t}}(h_{t,\theta})$ is unknown, it needs to be estimated (e.g. by one of the methods in \cite{glynn2002CV}, and we use the one described in eqs. (6) and (7) in \citet{Segers18CV}):
\begin{equation}\label{eq:pnthbest}
\begin{split}
\hat{\beta}_{H,\theta,f_t}(h_{t,\theta}) & = \{ P_{H,\theta}(h_{t,\theta} h_{t,\theta}^T) - P_{H,\theta}(h_{t,\theta})P_{H,\theta}(h_{t,\theta}^T) \}^{-1} \times \\
&  \quad\quad \{ P_{H,\theta}(h_{t,\theta} f_t) - P_{H,\theta}(h_{t,\theta}) P_{H,\theta}(f_t) \}.
\end{split}
\end{equation}
For the iid sample $\tilde\bsX_j(\theta), j = 1,\dots,H$, as in \eqref{eq:Xjhblcks} we obtain the {\em control variates approximation} of $\vp(\cdot,\theta)$ {given by}
\begin{equation}\label{eq:cv_cf}
 \vp^{\text{(cv)}}_H(t,\theta) = P_{H,\theta}(f_t) - \kappa_{H}(t,\theta),\quad t\in\R^p,
 \end{equation}
where 
\begin{equation}\label{eq:cvc}
\kappa_{H}(t,\theta) = (\hat{\beta}_{H,\theta,f_t}(h_{t,\theta}))^T P_{H,\theta}(h_{t,\theta}).
\end{equation}
Recall from \eqref{eq:chfH} that $P_{H,\theta}(f_t)=\vp_{H}(t,\theta)$, so we could simply replace $\vp_{H}(t,\theta)$ in \eqref{eq:defQ} by $\vp^{\text{(cv)}}_H(t,\theta)$ as given in \eqref{eq:cv_cf}. 
However, as we shall see in Section~\ref{se:mc_vs_cv}, the control variates approximation $\vp^{\text{(cv)}}_H(t,\theta)$ provides superior approximations of $\vp(t,\theta)$ only for values of $t$, for which $\var(\la t,\bstildeX_1(\theta) \ra)$ is small. Thus, we replace $\vp_{H}(t,\theta)$ in \eqref{eq:defQ} by a combination of $\vp_{H}(t,\theta)$ and $\vp^{\text{(cv)}}_H(t,\theta)$. 
More precisely, we propose the following {\em control variates based estimator}:
\begin{equation}\label{eq:thecv}
\hat{\theta}^{\text{(cv)}}_{n,H,k} = \text{argmin}_{\theta \in \Theta} Q^{\text{(cv)}}_{n,H,k}(\theta),
\end{equation}
where for appropriate $k > 0$,
\beam
& &\quad Q^{\text{(cv)}}_{n,H,k}(\theta) = \label{eq:Qcv}\\
& & \int_{\mathbb{R}^p} \bigg| \vp_n(t) - \bigg( \vp^{\text{(cv)}}_H(t,\theta) 1_{\{\wh{\var}(\la t,\bsX_1 \ra) < k\}} + \vp_{H}(t,\theta) 1_{\{\wh{\var}(\la t,\bsX_1 \ra) \geq k\}} \bigg) \bigg|^2  \bar{w}(t) \text{d} t, \nonumber
\eeam
$\bar{w}(t) = \frac{w(t)}{\wh{\var}(\la t,\bsX_1 \ra) }$, with suitable weight function $w$ such that the integral is well-defined.

{It is worth mentioning that, for a fixed weight function $w(\cdot)$, the weight function $\bar{w}(\cdot)$ can always be computed since $\wh{\var}(\la t,\bsX_1 \ra)$ depends only on the time series data. The downside of using the control variates based estimator \eqref{eq:thecv} is that one needs to resort to numerical integration. However, the procedure is feasible for moderate dimension $p$. As illustrated in the Poisson-AR example of Section~\ref{se:parm_cf}, the control variates based estimator has improved the performance over the simulation based estimator \eqref{eq:defThe} considerably.} 

Note that {$\wh{\var}(\la t,\bsX_1 \ra) = t^T \hat{\Gamma}_p t$} where $\hat{\Gamma}_p = (\hat{\gamma}_p (i-j))_{i,j=1}^p$ with
\begin{equation}\label{eq:gamma}
\hat{\gamma}_p (h) =  \frac{1}{n-h}\sum_{j=1}^{n-h} (X_j - \hat{\mu}_n)(X_{j+h} - \hat{\mu}_n), \quad h = 1,\dots,p,
\end{equation}
and $\hat{\mu}_n = \frac{1}{n} \sum_{j=1}^n X_j$. The choice of the indicator function $1_{\{\wh{\var}(\la t,\bsX_1 \ra) < k\}}$ is justified by the fact that, when estimating the parameter $\theta_0$, we focus on approximations of $\vp(t,\theta)$ for $\theta$ close to $\theta_0$.

\section{Asymptotic behavior of the parameter estimators}\label{se:asymp}

Before performing the parameter estimation we need to make sure that the parameters are identifiable from the model.

%
%
%
 

{In the following we assume that the model parameters are identifiable from the chf.
{In our examples, the dimension $p$ must be at least 2.
For a specific choice of $p$, the minimum in \eqref{eq:thecv} may not be unique giving an identifiability problem of the estimated model. This may be remedied by increasing the dimension $p$.}}

%

{In the sequel, we will make various assumptions on different aspects of the underlying process, smoothness of the model, moments of the process, and properties of the weight function. We group these assumptions into the following categories.}

\begin{assumptionA}[Parameter space and time series process]
\mbox{}
\begin{enumerate}[label=$(a.\arabic*)$]
\item $\Theta$ is a compact subset of $\R^q$ and
 $\theta_0 \in \Theta^{\mathrm{o}}$, the interior of $\Theta$.\label{as:compact}
\item $(X_j)_{j \in \Z}$ is a stationary and ergodic sequence\label{as:ergodXt}.
\item $(X_j)_{j \in \Z}$ is $\alpha$-mixing with rate function $(\alpha_j)_{j \in \N}$ satisfying $\sum_{j=1}^\infty (\alpha_j)^{1/r} < \infty$ for some $r > 1$\label{as:amixXt}.
\end{enumerate}
\end{assumptionA}

\begin{assumptionB}[Continuity and differentiability in $\theta$]
\mbox{}
\begin{enumerate}[label=$(b.\arabic*)$]
\item For each $j \in {\N}$, the map $\theta \mapsto {\bstildeX}_j(\theta)$ is continuous on $\Theta$.
\label{as:contXt}
\item For each $j \in {\N}$, the map $\theta \mapsto {\bstildeX}_j(\theta)$ is twice continuously differentiable in an open neighborhood around $\theta_0$\label{as:2diffXt}.
\end{enumerate}
\end{assumptionB}

\begin{assumptionC}[Moments]
\mbox{}
\begin{enumerate}[label=$(c.\arabic*)$]
\item $\E |X_1|^u < \infty$, where $u = 2r / (r-1)$ with $r > 1$ being such that \ref{as:amixXt} holds.\label{as:momX1u}
\item $\E \prod_{j=1}^p |X_j|^{\alpha} < \infty$ for some $\alpha \in (u/2,u]$ where $u = 2r / (r-1)$ with $r > 1$ being such that \ref{as:amixXt} holds\label{as:momTigh}.
\item $\E \suptheta | X_1(\theta)|^4 <\infty$\label{as:EsupX4}.
\item  For each $\theta \in \Theta$, $\E |\gradtheta X_1(\theta)| < \infty$. \label{as:mom1Der}
\item $\E \suptheta | \gradtheta X_1(\theta)|^{2(1+\eps)} <\infty$ and $\E \suptheta | \gradtheta^2 X_1(\theta)|^{1+\eps} < \infty$ for some $\eps > 0$\label{as:momSD}.
\end{enumerate}
\end{assumptionC}

\begin{assumptionD}[Weight function]
\mbox{}
\begin{enumerate}[label=$(d.\arabic*)$]
\item $\intrp w(t) \diff t < \infty$\label{as:intWt}.
\item $\intrp |t| w(t) \diff t < \infty$\label{as:inttWt}.
\item $\intrp |t|^{2(1+\eps)} w(t) \diff t < \infty$ for some $\eps > 0$\label{as:intt2Wt}.
\item $\int_{\R^p} \frac{w(t)}{|t|^2} \diff t < \infty$\label{as:insuptd}.
\end{enumerate}
\end{assumptionD}

Assumption~B is indeed satisfied by
many linear and non-linear time series processes, in particular, when they  have a representation $X_j(\theta) = f(Z_j,Z_{j-1},\cdots;\theta)$ or \\ $X_j(\theta) = f(Z_j,X_{j-1}(\theta), X_{j-2}(\theta),\cdots;\theta)$  for  iid noise variables $(Z_j)_{j \in \Z}$,  and $f:\R^\infty \times \Theta \mapsto \R$ is a measurable function.
Prominent examples are the MA$(\infty)$ and AR$(\infty)$ representations of a {causal or invertible} ARMA$(p,q)$ model (see e.g. eqs. (3.1.15) and (3.1.18) in \citet{Brockwell13}) or the ARCH$(\infty)$ representation of a GARCH~$(p,q)$
model (see e.g. \citet{Francq11}, Theorem 2.8).
In this case, assumptions \ref{as:contXt} and \ref{as:2diffXt} will hold whenever the map $f$ is continuously differentiable for $\theta\in\Theta$. For example, if $f$ is Lipschitz-continuous for $\theta\in\Theta$, 
then the continuity assumption \ref{as:contXt} holds. 


The key asymptotic properties, consistency and asymptotic normality of our estimates are stated in the following theorems.
The proofs of these results are presented in the Appendix.

We formulate first the strong consistency results of the parameters.

\begin{theorem}[Consistency of $\thehat$]\label{th:cons}
Assume that \ref{as:compact}, \ref{as:ergodXt}, \ref{as:contXt}, and \ref{as:intWt} hold. Let $H = H(n) \rightarrow \infty$ as $\nto$. Then {$\thehat \stas \theta_0$ as $\nto.$}
\end{theorem}

\begin{theorem}[Consistency of $\thehatcv$]\label{th:consCV}
Assume that the conditions of Theorem~\ref{th:cons} hold, and additionally \ref{as:momX1u}, \ref{as:EsupX4}, and \ref{as:insuptd}.
 Let $H = H(n) \rightarrow \infty$ as $\nto$. Then {$\thehatcv \stas \theta_0$ as $\nto.$}
\end{theorem}

The asymptotic normality of the simulation based parameter estimator reads as follows.

\begin{theorem}[Asymptotic normality of $\thehat$]\label{th:asn}
Assume that Assumptions A and B, and the moment conditions \ref{as:momTigh}, \ref{as:mom1Der}, and \ref{as:momSD} hold.
Furthermore, assume that the weight function satisfies \ref{as:intWt}, \ref{as:inttWt} and \ref{as:intt2Wt}.
Set $H=H(n) :=\bar H (n) n$ and $\bar H (n)\to\infty$ as $\nto$ and define
\begin{equation}\label{eq:kjdef}
K_j(\theta) =  \int_{\R^p} \Big( \frac{\partial}{\partial{\theta}} \Re(\varphi(t,\theta)), \frac{\partial}{\partial{\theta}} \Im(\varphi(t,\theta)) \Big) \begin{pmatrix}
 \cos(\la t, \bsX_j \ra) - \Re(\vp(t,\theta)) \\ 
  \sin(\la t, \bsX_j \ra) - \Im(\vp(t,\theta))
\end{pmatrix}w(t) \diff t, \quad j \in \N
\end{equation}
and
\begin{equation}\label{eq:rev3}
Q =  \intrp \Big( \frac{\partial}{\partial{\theta}} \Re (\varphi(t,\theta_0)) , \frac{\partial}{\partial{\theta}}  \Im (\varphi(t,\theta_0)) \Big)  \Big( \frac{\partial}{\partial{\theta}} \Re (\varphi(t,\theta_0)) , \frac{\partial}{\partial{\theta}}  \Im (\varphi(t,\theta_0)) \Big)^T  w(t) \diff t.
\end{equation}
If $Q$ is a non-singular matrix, then 
\begin{equation}\label{eq:anthe}
\sqrt{n}(\thehat - \theta_0) \std N(0,Q^{-1} W Q^{-1}), \quad \nto,
\end{equation}
where
\begin{equation}\label{eq:Wdeffor}
W = \var [K_1(\theta_0)] + 2 \sum_{j=2}^{\infty} \cov[ K_1(\theta_0), K_j(\theta_0) ].
\end{equation}
\end{theorem}

Theorem~\ref{th:asn} shows that $\thehat$ is asymptotically normal {and achieves the same asymptotic efficiency as the oracle estimator} from \eqref{eq:defOra} {(see Theorem~2.1 in \cite{Knight02Corr})}. 
Therefore, there cannot be any improvement in the limit law for the asymptotic normality of $\thehatcv$. 
However, as we show in Section~\ref{se:mc_vs_cv}, $\thehatcv$ is based on a better approximation of the chf $\vp(\cdot,\theta)$ than that used for $\thehat$. 
Thus, the control variates estimator $\thehatcv$ improves the finite sample performance  compared to the simulation based estimator $\thehat$. 

\begin{remark}
As pointed out in \cite[Remark~2.3]{Knight02Corr},  the asymptotic variance of $\thehat$ in 
\eqref{eq:anthe} can be approximated by replacing $\theta_0$ by $\thehat$ in \eqref{eq:rev3} and \eqref{eq:Wdeffor} and by replacing the infinite sum in \eqref{eq:Wdeffor} by an approximating sum with a kernel and a convenient bandwidth using the methods suggested in \cite{andrews1991heteroskedasticity} and \cite{newey1994automatic}.
\end{remark}

}

\section{Assessing the quality of the estimated chf}\label{se:mc_vs_cv}

In this section we compare the performance of both the Monte Carlo approximation $\vp_{H}(\cdot,\theta)$ and the control variates approximation $\vp^{\text{(cv)}}_H(\cdot,\theta)$ of the chf as defined in \eqref{eq:ecfsim} and \eqref{eq:cv_cf}, respectively. We start with the following comparison of the two chf approximations.

\begin{remark}\label{re:compcvth}[Comparison of $\vp^{\text{(cv)}}_H(\cdot,\theta)$ and $\vp_H(\cdot,\theta)$]
Assume that \ref{as:EsupX4} holds, and let $\vp^{\text{(cvopt)}}_H$ and $\vp^{\text{(cv)}}_H$ be as defined in \eqref{eq:CVoracle} and \eqref{eq:cv_cf}, respectively. 
We use that
$\hat{\beta}_{H,\theta,f_t}(h_{t,\theta}) \stas \beta^{(\text{opt})}_{\theta,f_{t}}(h_{t,\theta})$ as $\nto$ with limit given in \eqref{eq:dfbeopt}.
This follows from the representation of $\hat{\beta}_{H,\theta,f_t}(h_{t,\theta})$ as 
\begin{equation*}
\hat{\beta}_{H,\theta,f_t}(h_{t,\theta}) = \hat{\beta}_{H,\theta,\Re(f_t)}(h_{t,\theta}) + i \hat{\beta}_{H,\theta,\Im(f_t)}(h_{t,\theta})
\end{equation*}
and the almost sure convergence of both terms. The quantities needed to compute the estimator in \eqref{eq:pnthbest} are, for each $\nu,\kappa = 1,2$:
\beam
P_{H,\theta}(f_t ) & = &  \frac1H  \sum_{j=1}^H  e^{i \la t,\bstildeX_j(\theta) \ra} ,  \label{eq:CVQ1}\\
P_{H,\theta}(h_{\nu,t,\theta}) & = &  \frac1H  \sum_{j=1}^H  \Big( \la t,\bstildeX_j(\theta) \ra^\nu - \E \la t,\bsX_1(\theta) \ra ^\nu \Big),\nonumber\\
P_{H,\theta}(f_t h_{\nu,t,\theta}) & = &  \frac1H  \sum_{j=1}^H e^{i\la t,\bstildeX_j(\theta) \ra}  \Big( \la t,\bstildeX_j(\theta) \ra^\nu - \E  \la t,\bsX_1(\theta) \ra^\nu  \Big),\nonumber\\
P_{H,\theta}(h_{\nu,t,\theta}h_{\kappa,t,\theta}) & = &  \frac1H \sum_{j=1}^H  \Big( \la t,\bstildeX_j(\theta) \ra^\nu - \E \la t,\bsX_1(\theta) \ra^\nu \Big)   \nonumber\\
& & \quad \times\Big( \la t,\bstildeX_j(\theta) \ra^\kappa - \E \la t,\bsX_1(\theta) \ra^\kappa \Big).\label{eq:CVQ4}
\eeam
Hence, strong consistency of $\hat{\beta}_{H,\theta,f_t}(h_{t,\theta})$ follows from the SLLN.
This together with $P_\theta(h_{t,\theta}) = 0$ implies by Theorem~1 in \citet{glynn2002CV} that, as $H \rightarrow \infty$,
\begin{equation*}
H^{1/2}\big( \Re\big( \vp^{\text{(cv)}}_H(t,\theta) - \vp(t,\theta)\big)\big) \std N\big(0, \sigma_\theta^2\big( \Re(f_t) - [\beta^{(\text{opt})}_{\theta,\Re(f_t)}(h_{t,\theta})]^T h_{t,\theta}\big) \big),
\end{equation*}
\begin{equation*}
H^{1/2}\big( \Im\big( \vp^{\text{(cv)}}_H(t,\theta) - \vp(t,\theta)\big)\big) \std N\big(0, \sigma_\theta^2\big( \Im(f_t) - [\beta^{(\text{opt})}_{\theta,\Im(f_t)}(h_{t,\theta})]^T h_{t,\theta}\big) \big),
\end{equation*}
with 
$$
\sigma_\theta^2\big( \Re(f_t) - [\beta^{(\text{opt})}_{\theta,\Re(f_t)}(h_{t,\theta})]^T h_{t,\theta}\big) \leq \sigma_\theta^2\big( \Re(f_t) \big) \quad \text{and} \quad
$$
$$
\sigma_\theta^2\big( \Im(f_t) - [\beta^{(\text{opt})}_{\theta,\Im(f_t)}(h_{t,\theta})]^T h_{t,\theta}\big) \leq \sigma_\theta^2\big( \Im(f_t) \big),
$$
with $\sigma^2_{\theta}(\cdot)$ as defined in \eqref{eq:defst2}. Therefore, $\vp^{\text{(cv)}}_H(\cdot,\theta)$ provides an approximation of the integral $Q_n(\theta)$ in \eqref{eq:defQast} with smaller variance than $\vp_{H}(\cdot,\theta)$. As a consequence, this favors the control variates estimator $\thehatcv$ over the simulation based estimator $\hat{\theta}_{n,H}$ for large sample sizes $n\in\N$. \hspace*{5cm}
\end{remark}

For all forthcoming examples we choose $p =3$ and $H = 3\,000$. We begin with a stationary Gaussian AR(1) process, where we know the chf $\vp(\cdot)$ explicitly,  and then proceed to the Poisson-AR process, where we approximate the true unknown chf by a precise simulated version.

\subsection{The {Gaussian} AR(1) process}\label{se:armcf}

We start with a stationary Gaussian AR(1) process to show how the method of control variates improves the Monte Carlo approximation of its chf.
Let $(X_j(\theta))_{j \in \Z}$ be the AR(1) process 
\begin{equation}\label{eq:defar}
X_j(\theta) = \phi X_{j-1}(\theta) + Z_j(\theta), \,  j \in \Z ,\quad (Z_j(\theta))_{j \in \Z} \simiid N(0,\sigma^2),
\end{equation}
with parameter space $\Theta$ being a compact subset of $\{\theta = (\phi,\sigma) : |\phi| < 1,\sigma > 0\}$. 
Then the true chf of $\bsX_1(\theta) = (X_1(\theta),X_2(\theta),X_3(\theta))$ is given by {$\vp(t,\theta)= e^{-\frac12 t^T\Gamma_3(\theta) t}$ for $t\in\R^3,$}
where the covariance matrix $\Gamma_3(\theta)$ is explicitly known and identifies the parameter $\theta$ uniquely; see e.g. \cite{Brockwell13}, Example~3.1.2.
For a fixed $\theta \in \Theta$ and many $t \in \R^3$ we compute the absolute errors
\begin{equation}\label{eq:errort}
\xi_H(t,\theta) = |\vp_{H}(t,\theta) - \vp(t,\theta)|\,\,\mbox{and}\,\, \xi_H^{\text{(cv)}}(t,\theta) = |\vp_{H}^{\text{(cv)}}(t,\theta) - \vp(t,\theta)|
\end{equation}
where $\vp_{H}(\cdot,\theta)$ is the Monte Carlo approximation of the chf of $\bsX_1(\theta) = (X_1(\theta),X_2(\theta),X_3(\theta))$ and $\vp_{H}^{\text{(cv)}}(\cdot,\theta)$ its control variates approximation.
To understand how well we can approximate $\vp(\cdot,\theta)$, we plot in Figure~\ref{fig:ar1_abst_and_vart}, $\xi_H(t,\theta)$ and $\xi_H^{\text{(cv)}}(t,\theta)$ against $\sqrt{\var[ \la t,\bsX_1(\theta) \ra]}$ for different parameters $\theta$.
These quantities are computed from an iid sample $\bsX_j(\theta), j = 1,\dots,H$ as in \eqref{eq:Xjhblcks}. 
To simulate iid observations from the model \eqref{eq:defar}, we use the fact that the one-dimensional stationary distribution is
$X_1(\theta) \sim N(0, \sigma^2 / (1 - \phi^2))$,
 and then use the recursion in \eqref{eq:defar} to simulate $X_2(\theta)$ and $X_3(\theta)$. 
We chose $500$ randomly generated values of $t$ from the $3$-dimensional Laplace distribution with chf given in \eqref{eq:wtLap}.

It is clear from Figure~\ref{fig:ar1_abst_and_vart} that both the Monte Carlo and the control variates approximations work better when $\sqrt{\var[ \la t,\bsX_1(\theta) \ra]}$ is small, and also that the control variates approximations are best for small values of $\sqrt{\var[ \la t,\bsX_1(\theta) \ra]}$.
The superiority of the control variates approximation for all $t$ and all parameter settings is clearly visible, and
 already expected from Remark~\ref{re:compcvth}. 

\subsection{The Poisson-AR model}\label{se:parm_cf}

We consider a nonlinear time series process for time series of counts, which has been proposed originally in \citet{Zeger88}. 
A prototypical Poisson-AR(1) model suggested in \citet{DavisRod} assumes that the observations $(X_j(\theta))_{j \in \Z}$ are independent and Poisson-distributed with means $e^{\beta + \alpha_j(\theta)}$ where the process $(\alpha_j(\theta))_{j\in\Z}$ is a latent stationary Gaussian AR(1) process, given by the equations
\begin{equation*}
\alpha_j(\theta) = \phi \alpha_{j-1}(\theta) + \eta_j(\theta), \, j\in\Z, \quad (\eta_j(\theta))_{j \in \Z} \simiid N(0,\sigma^2),
\end{equation*}
with parameter space $\Theta$ being a compact subset of  $\{\theta = (\beta,\phi,\sigma) :|\phi| < 1, \beta\in\R, \sigma > 0\}$. 
The parameter $\theta$ is uniquely identifiable from the second order structure, which has been computed in Section~2.1 of \citet*{Davis00PRM}.

For this model, the true chf of $\bsX_1(\theta) = (X_1(\theta),X_2(\theta),X_3(\theta))$ cannot be computed in closed form. 
To mimic the assessment of the errors in eq. \eqref{eq:errort}, we simulate $1\,000\,000$ iid observations from $\bsX_1(\theta)$
by first simulating a Gaussian AR(1) process $(\alpha_1(\theta),\alpha_2(\theta),\alpha_3(\theta))$ (as described in Section~\ref{se:armcf}) and then simulating independent Poisson random variables with means $e^{\beta+\alpha_1(\theta)}$, $e^{\beta+\alpha_2(\theta)}$ and $e^{\beta+\alpha_3(\theta)}$, respectively.
From this we compute the empirical characteristic function and take it as $\vp(\cdot,\theta)$ in the absolute error terms \eqref{eq:errort}. 

We compare the performance of both the Monte Carlo approximation and the control variates approximation of the chf. Figure~\ref{fig:par} presents the results. The plots in Figure~\ref{fig:par} are also in favor of the control variates approximation, when compared to the Monte Carlo approximation.

\noindent%
\begin{minipage}{\linewidth}
\makebox[\linewidth]{
  \includegraphics[width = 14cm, height=23cm]{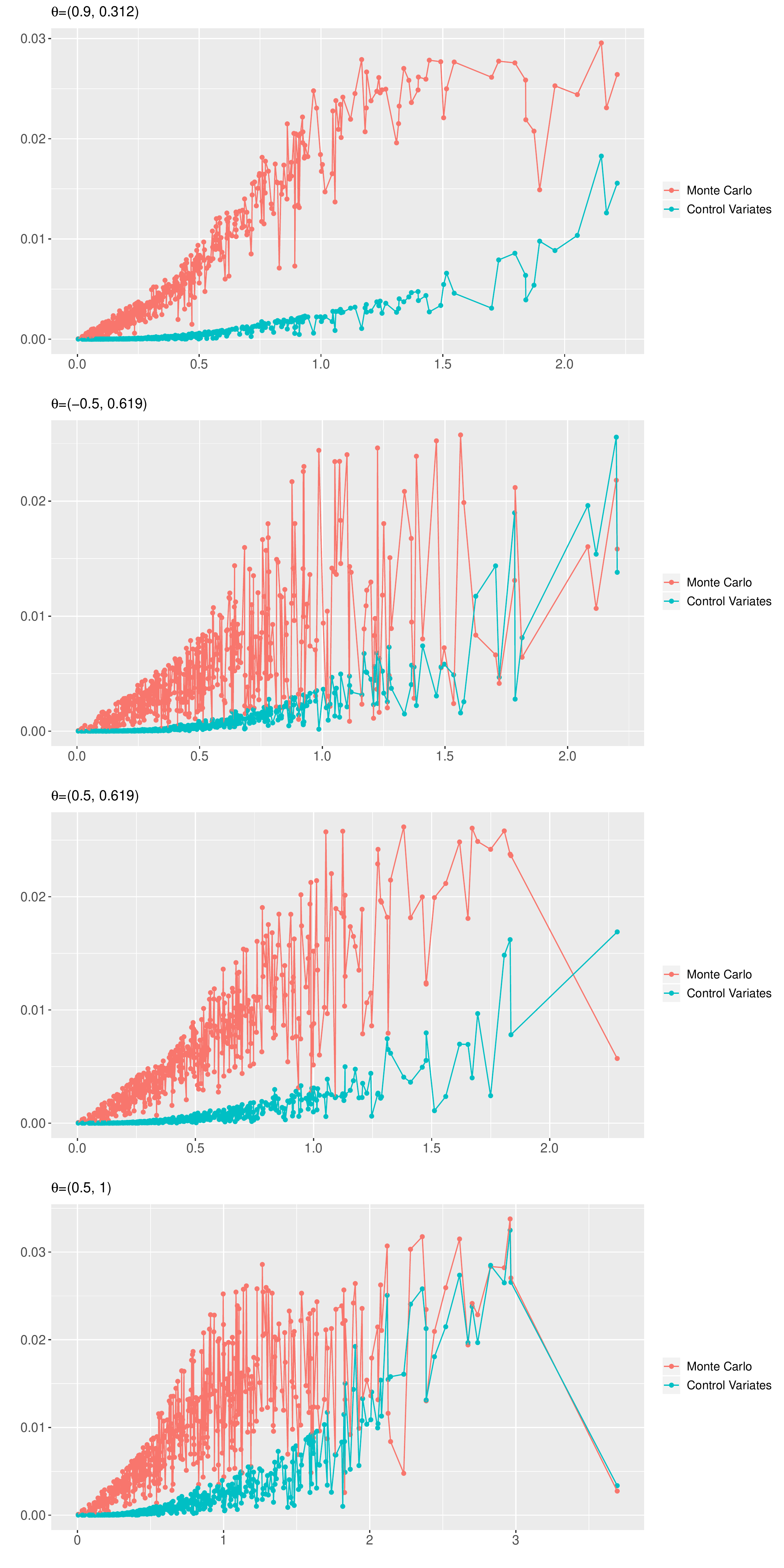}}
\captionof{figure}{\footnotesize Gaussian AR(1) model: absolute error $\xi_H(t,\theta)$ {(red)} and $\xi_H^{\text{(cv)}}(t,\theta)$ {(green)} for $p=3$ and $H=3\,000$ as in eq. \eqref{eq:errort}. 
We use $500$ randomly generated values of $t \in \R^3$ from the Laplace distribution (with chf as in \eqref{eq:wtLap} below), which are plotted against $\sqrt{\var[ \la t,\bsX_1(\theta) \ra]}$.}\label{fig:ar1_abst_and_vart}
\end{minipage}


\noindent%
\begin{minipage}{\linewidth}
\makebox[\linewidth]{
  \includegraphics[width = 14cm, height=23cm]{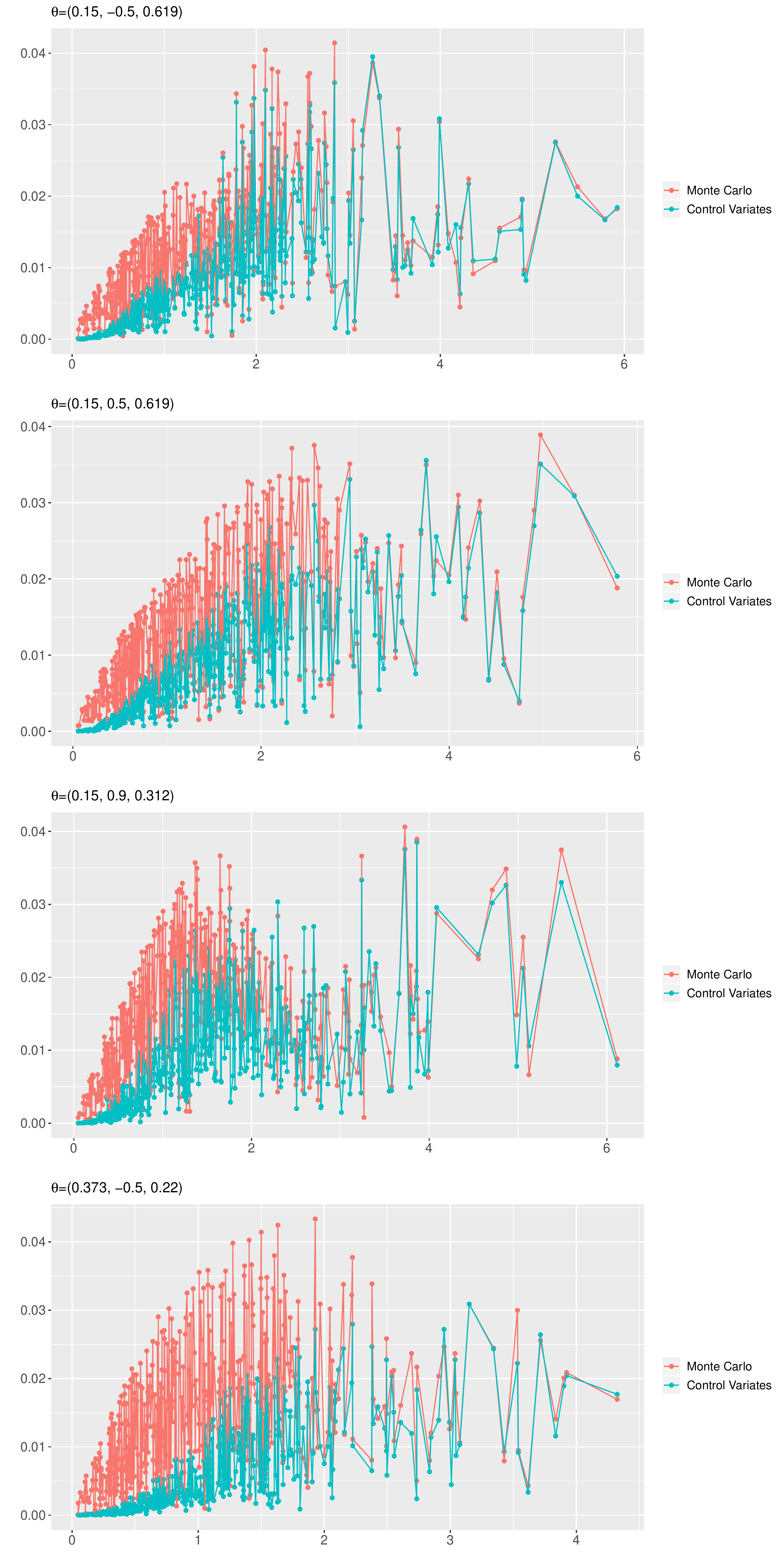}}
\captionof{figure}{\footnotesize Poisson-AR model: Absolute errors $\xi_H(t,\theta)$ {(red)} and $\xi_H^{\text{(cv)}}(t,\theta)$ {(green)} for $p=3$ and $H=3\,000$ as in eq. \eqref{eq:errort}. We use $500$ randomly generated values of $t \in \R^3$ from the Laplace distribution (with chf as in \eqref{eq:wtLap} below), which are plotted against $\sqrt{\var[ \la t,\bsX_1(\theta) \ra]}$.}\label{fig:par}
\end{minipage}

\section{Practical aspects and simulation results}\label{s5}

Our objective is to obtain a simple expression of the integrated mean squared error $Q_{n,H}(\theta)$ in \eqref{eq:defQ}, which is needed to compute the estimator in \eqref{eq:defThe}. 
For a weight function $w$ in \eqref{eq:defQ}, we write
\begin{equation}\label{eq:what}
\tilde{w}(x) = \int_{\R^p} e^{i\la t, x \ra } w(t) \diff t,\quad x \in \R^p,
\end{equation}
for its Fourier transform. 
Our preference is on weight functions such that \eqref{eq:what} is known explicitly.

\bexam[Weight functions and their characteristic functions]\\
(i) \, Laplace: $w$ is a multivariate Laplace density with chf
\begin{equation}\label{eq:wtLap}
\tilde{w}(t) = \frac{1}{(1 + (2\pi^2)^{-1} \,t^T t )}, \quad t \in \R^p.
\end{equation}
(ii) \, Cauchy: $w$ is  a multivariate Cauchy density with chf 
\begin{equation*}
\tilde{w}(t) = e^{-\sqrt{t^T t }}, \quad t \in \R^p.
\end{equation*}
(iii) \, Gaussian: $w$ is a standard multivariate Gaussian density with chf
\begin{equation}\label{eq:wtGau}
\tilde{w}(t) = e^{-\frac12 t^T t }, \quad t \in \R^p.
\end{equation}
\eexam


\begin{lemma}\label{lem:chfexpl}
Let $Q_{n,H}(\theta)$ be as in \eqref{eq:defQ} and $w$ a weight function with Fourier transform $\tilde w$. 
Then
\beam\label{eq:Qfour}
Q_{n,H}(\theta) 
& = &\frac{1}{n^2} \sum_{k=1}^{n}\sum_{j=1}^{n} \tilde{w}(\bsX_j-\bsX_k) + \frac{1}{H^2} 
\sum_{j=1}^{H}\sum_{k=1}^{H} \tilde{w}(\bstildeX_{j}(\theta)-\bstildeX_{k}(\theta))\nonumber\\
&& -\frac{1}{Hn} \sum_{k=1}^{H}\sum_{j=1}^{n} \Big(\tilde{w}(\bsX_j-\bstildeX_k(\theta))  + \tilde{w}(\bstildeX_k(\theta)-\mathbf{X}_j)\Big).
\eeam
\end{lemma}

Formula \eqref{eq:Qfour} is very useful, since it avoids the computation of a $p$-dimensional integral. 
Additionally, since the first double sum on the right-hand side of \eqref{eq:Qfour} does not depend on the argument $\theta$,
for the optimization it can be ignored.



\begin{remark}
{When evaluating the integrated weighted mean squared errors  \eqref{eq:defQ}, \eqref{eq:Qcv}, or \eqref{eq:Qfour} in practice,
they need to  be deterministic functions of $\theta$. This is enforced by taking a fixed seed for every $j = 1,\dots,H$, when simulating $\bstildeX_j(\theta)$ for different values of $\theta \in \Theta$.}
\end{remark}

In the following two examples we study the finite sample behavior of the estimators $\thehat$ and $\thehatcv$. We begin with a stationary Gaussian ARFIMA model, whose chf is explicitly known so that we can use the oracle estimator from Section~\ref{s21}. Afterwards we come back to the Poisson-AR process.
We choose $p=3$, since the 3-dimensional chf contains sufficient information to identify the parameter of interest.
We also choose $H=3\,000$. 

\subsection{The ARFIMA model}\label{s52}

Let $(X_j(\theta))_{j \in \Z}$ be the stationary Gaussian ARFIMA$(0,d,0)$ model
\begin{equation*}
(1-B)^d X_j(\theta) = Z_j(\theta), \, j\in\Z, \quad (Z_j(\theta))_{j \in \Z} \simiid N(0,\sigma^2),
\end{equation*}
where $B$ is the backshift operator, with parameter space $\Theta$ being a compact subset of $\{\theta=(d,\sigma) : d \in (-0.5, 0.5), \sigma > 0\}$. 
Then the true chf of $\bsX_1(\theta) = (X_1(\theta),X_2(\theta),X_3(\theta))$ is given by {$\vp(t,\theta)=e^{-\frac12 t^T\Gamma_3(\theta)t}$ for $t\in\R^3, \theta\in\Theta,$}
where the covariance matrix $\Gamma_3(\theta)$ is explicitly known and identifies the parameter $\theta$ uniquely; see e.g. \citet{Taqqu17Book},  Corollary~2.4.4. 

For the long-memory case, for each value of $d \in \{0.05,\dots,0.45\}$ we compare the new estimators 
with the MLE method as implemented in the R package \texttt{arfima}. 
Thus, for many $\theta \in \Theta$, we generate iid Gaussian random vectors with mean zero and covariance $\Gamma_3(\theta)$ and use them to construct the simulation based estimator $\thehat$.

Since the chf $\vp(\cdot,\theta)$ is known in closed form, we are able to compute the oracle estimator $\hat{\theta}_n$ from \eqref{eq:defQast}. 
In order to compute the integral appearing in \eqref{eq:defQast} in closed form, we choose the weight function $w(t) = (2\pi)^{-3/2} e^{-\frac{1}{2} t^T t}, t\in\R^3$.

Then the integral in \eqref{eq:defQast}, which needs to be minimized with respect to the parameter $\theta$, can be evaluated similarly as in \eqref{eq:Qfour}, giving for the chf being known, {that $Q_n(\theta)$ can be written as}
\beam\label{eq:QAstLast}
&& \int_{\R^3} \bigg| \frac{1}{n}\sum_{j=1}^n e^{i\la t,\bsX_j \ra} - e^{-\frac{1}{2}t^T \Gamma_3(\theta) t}\bigg|^2 w(t) \diff t = \big(\det\big((2\Gamma_3(\theta) + I)^{-1}\big)\big)^{\frac{1}{2}}\nonumber\\
 &  & \quad\quad + \frac{1}{n^2} \sum_{j=1}^n \sum_{k=1}^n \exp\Big\{-\frac{1}{2}(\bsX_j - \bsX_k)^T(\bsX_j-\bsX_k)\Big\} 
   \\
& &\quad\quad -2 \, \big(\det\big((\Gamma_3(\theta) + I)^{-1}\big)\big)^{\frac{1}{2}} \frac{1}{n} \sum_{j=1}^n \exp\Big\{-\frac{1}{2}\bsX_j^T(\Gamma_3(\theta) + I)^{-1}\bsX_j\Big\}\nonumber. 
\eeam
We compare in Table~\ref{tb:arfimaGa} the performance of the simulation based estimator $\thehat$, the oracle estimator $\hat{\theta}_n$ in \eqref{eq:defOra} based on the minimization of \eqref{eq:QAstLast},
and the MLE. We fixed $\sigma=1$ for all simulated sample paths used in the simulation study. For both $\hat{\theta}_n$ and $\hat{\theta}_{n,H}$, we also estimate $\sigma$ but report only the performance for the estimator of $d$ which is the key parameter of interest in long-range dependence models.
We notice that $\thehat$ 
 is comparable to the oracle estimator, so in this model there is no need to use control variates. When comparing both simulation based estimators, the RMSEs are almost the same for all $d \geq 0.20$. 
The MLE has a smaller RMSE, but both $\hat{\theta}_n$ and $\thehat$ have a smaller bias than the MLE. In the simulations, the density plots for the estimates of $d$ with $d \in \{0.25,0.3\}$ look reasonably normal. On the other hand, the estimates when $d$ is closer to $0.5$ are rather skewed, which is expected due to the constraint $d < 0.5$. In this case a larger sample is needed in order to obtain more normal looking densities.

\begin{remark}
 We also investigate the feasibility of our new estimation procedures for misspecified models.  We take a Gaussian ARFIMA as the true model, but for the data we modify the distribution of its innovations. Specifically, we consider the two cases of ARFIMA models driven by noise with a Laplace distribution and with a Student-$t$ distribution with 6 degrees of freedom. The estimation results under the two misspecification scenarios are shown in Tables~\ref{tb:arfimadExp} and \ref{tb:arfimadStudent} of the Appendix. 
 The quasi-oracle estimator is based on the Gaussian chf, and 
the quasi-MLE (QMLE) is found by maximizing the Gaussian likelihood, even though the data are in fact nonGaussian.  For both noise distributions, we see very little difference in the performance of the three estimators (QMLE compared with MLE) from the Gaussian ARFIMA scenario in Table~\ref{tb:arfimaGa}.  In particular, our estimator continues to have small bias and RMSE that is comparable to the oracle estimator and only slightly larger than that of the QMLE.  Of course, it is known that the QMLE estimators behave asymptotically the same as the MLE when the data is Gaussian.   
\end{remark}

\subsection{The Poisson-AR process}\label{s51}

The Poisson-AR model has been defined in Section~\ref{se:parm_cf}. 
We conduct a simulation experiment in the same setting as in Table~5 in  \citet{DavisRod} and Table~3 in \citet{Davis11PL}. 
The results are shown in Table~\ref{tb:pam} of the Appendix for $n = 400$ and nine different parameter settings, where we also classify the models by the corresponding index of dispersion $D$ of the random variable $e^{\beta + \alpha_1}$, which assumes values in $\{0.1,1,10\}$ as shown in \citet{DavisRod}.

We compare both the simulation based estimator $\thehat$ and control variates based estimator $\thehatcv$. We fix $H = 3\, 000$, $p = 3$ and  the $3$-dimensional Laplace density as in \eqref{eq:wtLap} for $w$. 
To simulate iid observations of $(X_1(\theta),X_2(\theta),X_3(\theta))$ we proceed as explained in Section~\ref{se:parm_cf}. 
The simulation based estimator $\thehat$ in \eqref{eq:defThe} is computed via \eqref{eq:Qfour}. 
Unfortunately, such a formula cannot be obtained for the control variates based estimator $\thehatcv$, since the introduction of the correction $\kappa_H$ in \eqref{eq:cvc} introduces additional polynomial terms into $Q_{n,H,k}^{\text{(cv)}}$ in \eqref{eq:Qcv}. 
Thus, we resort to numerical integration to evaluate $\thehatcv$. 


Our findings are as follows.
For $D \in \{1,0.1\}$, the control variates based estimator $\hat{\theta}^{\text{(cv)}}_{n,H,k}$ for {$k=1$} presents smaller bias and RMSE than the simulation based estimator $\thehat$ in most cases, in all others it is comparable. {The smallest RMSE values are shaded in {Table~4}.
Additionally, a significant improvement in the bias for estimating $\phi$ is noticeable for $\theta = (0.373,0.500, 0.220)$ and $\theta = (0.373,0.900, 0.111)$. 
{This example shows the advantage of using control variates to improve the estimation of the model parameters.   This is not surprising in view of the improved performance of estimating the characteristic function as seen in all three panels of Figure~2.}

We compare now the control variates based estimator $\hat{\theta}^{\text{(cv)}}_{n,H,k}$ in Table~2 of the Appendix, with the results for the consecutive pairwise likelihood (CPL) from Table~3 in \citet{Davis11PL}, which is referred to as CPL1 in that paper. 
The bias of $\hat{\theta}^{\text{(cv)}}_{n,H,k}$ is smaller than that of CPL1 for the estimated $\beta$ and $\sigma$ for almost all cases, in all others it is comparable.
For $\phi$ the bias of $\hat{\theta}^{\text{(cv)}}_{n,H,k}$ and CPL1 are comparable, except that $\hat{\theta}^{\text{(cv)}}_{n,H,k}$ {shows} poor performance for estimating $\phi$ for the true parameter $(\beta,\phi,\sigma) = (0.373,0.9,0.111)$. 
This is due to the fact that the simulated sample paths contain a large number of zeros, giving very little information for the parameter estimation. The estimated values for $\beta$ look normal for all parameter choices. The sampling distributions of the other parameter estimates look close to normal, except in the boundary. In particular, the density for the estimates of $\phi$ when $\phi = 0.9$ or $\sigma \in \{0.22,0.111\}$ and estimates of $\sigma$ when $\sigma \in \{0.22,0.111\}$ show some asymmetry, deviating from normality. This is not unexpected because they are close to the boundary.

%
%

\appendix 

\section{Appendix}

Here we present the proofs of 
the main Theorems, as well as tables of results on the simulation study.
Then, in Section~\ref{se:prof_sec4} we provide the proofs of Theorems~\ref{th:cons}, \ref{th:consCV}, and \ref{th:asn}. 
Finally, we present in Section~\ref{se:tables} the tables summarizing the finite sample behavior of the simulation based estimators for {ARFIMA models driven by noise from Gaussian, Laplace, and Student-$t$ distributions,} and the Poisson-AR(1) model discussed in Section~\ref{s5}.
\par

\subsection{Proofs of the main results}\label{se:prof_sec4}

In the following we define $H = H(n)$ and $\bar H =\bar H (n)=H(n)/n$, but omit the argument $n$ for notational simplicity.
Throughout the letter $c$ stands for any positive constant independent of the respective argument.
Its value may change from line to line, but is not of particular interest.
For a matrix with only real eigenvalues $\lambda_{\min}(\cdot)$ denotes the smallest eigenvalue.

We often use the uniform SLLN, which guarantees for a continuous stochastic process $(Z(t))_{t\in \R^p}$ satisfying $\E \sup_{t\in K} |Z(t)|<\infty$ that  \\$\sup_{t\in K} |Z(t)-\E Z(t)|\stas 0$ as $\nto$ for every compact set $K\subset\R^p$.
More precisely, we use the SLLN on the separable Banach space $C(K)$, the space of continuous functions on the compact set $K \subset \R^p$, endowed with the sup norm  (see e.g. Theorem~16(a) in \citet{Ferg} or Theorem~9.4 in \citet{Part}).

\noindent
\textbf{Proof of Theorem~\ref{th:cons}}: Let
\begin{equation*}
 Q(\theta) =  \intrp \big|  \vp(t,\theta_0)  -  \vp(t,\theta)  \big|^2 w(t) \diff t
\end{equation*}
be the candidate limiting function of  $Q_{n,H}(\theta)$. For $\delta > 0$ define the set
\begin{equation}\label{eq:def:kdel}
K_{\delta} = \{t \in \R^p: |t| \leq \delta\}.
\end{equation}
Since $| e^{i \la t,\bstildeX_1(\theta) \ra }| = 1$ for all $\theta$ and $t$, 
and the random elements $(\tilde\bsX_j(\theta), \theta \in \Theta)_{j=1}^{\infty}$ are iid,
the uniform SLLN holds giving 
\begin{equation}\label{eq:ecfsup}
\sup_{(t,\theta) \in \Theta \times K_\delta}  \bigg| \frac{1}{H} \sum_{j=1}^H e^{i \la t,\bstildeX_j(\theta) \ra }- \vp(t,\theta)    \bigg| \stas 0, \quad \nto.
\end{equation}
In particular, for $\theta=\theta_0$ we also have
\begin{equation}\label{eq:ecfsup2}
\sup_{t \in  K_\delta}  \bigg| \frac{1}{n} \sum_{j=1}^n e^{i \la t,\bsX_j \ra }- \vp(t,\theta_0)    \bigg| \stas 0, \quad \nto.
\end{equation}
Applying the inequality $||a|^2 - |b|^2| \leq 2|a-b|$ for  $a,b \in \C, |a|,|b| \leq 1$ gives
\begin{equation}\label{eq:con_2}
\begin{split}
&  |Q_{n,H}(\theta) - Q(\theta)| \\
&   = \intrp \bigg|  \Big|  \frac{1}{n}\sum_{j=1}^n e^{i \la t,\bsX_j \ra } -  \frac{1}{H} \sum_{j=1}^H e^{i \la t,\bstildeX_j(\theta) \ra }\Big|^2 -  
 |  \vp(t,\theta_0)   -  \vp(t,\theta)   |^2 \bigg| w(t) \diff t \\
& \leq 2 \intrp  \bigg|   \frac{1}{n}\sum_{j=1}^n e^{i \la t,\bsX_j \ra }  - \vp(t,\theta_0)   + 
\vp(t,\theta)  -  \frac{1}{H} \sum_{j=1}^H e^{i \la t,\bstildeX_j(\theta) \ra }    \bigg| w(t) \diff t  \\
 & \leq 2 \intrp \bigg\{   \Big| \frac{1}{n}\sum_{j=1}^n e^{i \la t,\bsX_j \ra }  - \vp(t,\theta_0)  \Big|  + 
  \sup_{\theta \in \Theta}  \Big| \vp(t,\theta)  -   \frac{1}{H} \sum_{j=1}^H e^{i \la t,\bstildeX_j(\theta) \ra } \Big|   \bigg\} w(t) \diff t \\
 & \leq 2  \sup_{(t,\theta) \in \Theta \times K_\delta} \bigg\{   \Big| \frac{1}{n}\sum_{j=1}^n e^{i \la t,\bsX_j \ra }  - \vp(t,\theta_0)  \Big|  + 
   \Big| \vp(t,\theta)  -   \frac{1}{H} \sum_{j=1}^H e^{i \la t,\bstildeX_j(\theta) \ra } \Big|   \bigg\}  \\  
  & \quad \times \int_{K_\delta} w(t) \diff t + 8 \int_{K^c_\delta} w(t) \diff t .
\end{split}
\end{equation}
Applying $\suptheta$ on both sides of \eqref{eq:con_2}, using \eqref{eq:ecfsup} combined with \ref{as:intWt}, and taking the limit for $\delta \downarrow 0$ gives
\begin{equation}\label{eq:Q_n_lim}
 \sup_{\theta \in \Theta}  |Q_{n,H}(\theta) - Q(\theta)| \stas 0, \quad \nto.
\end{equation}
Now we prove that $Q(\theta) = 0$ if and only if $\theta = \theta_0$. Obviously $Q(\theta_0) = 0$. 
If $\theta \not= \theta_0$, then 
the distributions of $\bsX_1$ and $\bstildeX_1(\theta)$ are different and thus also their characteristic functions are different. 
Since characteristic functions are continuous, it follows that they are different at least on an interval with positive Lebesgue measure; hence $Q(\theta) > 0$. 
Therefore, $Q(\theta)$ is uniquely minimized at $\theta_0$ and this fact together with \eqref{eq:Q_n_lim} gives strong consistency of $\thehat$.

\noindent
\textbf{Proof of Theorem~\ref{th:consCV}}: We have that $\wh{\var}(\la t,\bsX_1\ra) = t^T \hat{\Gamma}_p t$, with $\hat{\Gamma}_p$ being the $p$-dimensional empirical covariance matrix of the observed time series $(X_1,\dots,X_T)$ as in \eqref{eq:gamma}.  
Let $k > 0$ be fixed and
\begin{equation*}
 Q^{\text{(cv)}}(\theta) =  \intrp \big|  \vp(t,\theta_0)  -  \vp(t,\theta)  \big|^2 \frac{w(t)}{t^T \Gamma_p t } \diff t
\end{equation*}
be the candidate limiting function of $Q^{\text{(cv)}}_{n,H,k}(\theta)$ in \eqref{eq:Qcv}, where $\Gamma_p$ is the theoretical $p$-dimensional covariance matrix of the time series  process $(X_j)_{j \in \Z}$. 

Based on the definition of  $Q^{\text{(cv)}}_{n,H,k}(\theta)$ in \eqref{eq:Qcv}, we divide the domain of integration in the integrated mean squared error $|Q^{\text{(cv)}}_{n,H,k}(\theta)-Q^{\text{(cv)}}(\theta)|$ into 
$\{\wh{\var}(\la t,\bsX_1 \ra) < k\}$ and $\{\wh{\var}(\la t,\bsX_1 \ra) \geq k\}$, equivalently into $L_n = \{t \in \R^p: t^T \hat{\Gamma}_p t < k \}$ and its complement $L_n^c$. 

Recall also \eqref{eq:cv_cf} and \eqref{eq:cvc}.
Using $|e^{ix}| = 1$ for all $x \in \R$, together with $|ab - cd| \leq |b||a-c| + |c||b-d|$ for $a,b,c,d \in \C$ gives for the integral on $L_n^c$:
\beao
& & |Q^{\text{(cv)}}_{n,H,k}(\theta)-Q^{\text{(cv)}}(\theta)|_{L_n^c}  \\
 & & :=  \int_{L_n^c} \Bigg| \bigg| \frac{1}{n} \sum_{j=1}^n e^{i \la t,\mathbf{X}_j \ra } 
 -\frac{1}{H} \sum_{j=1}^H e^{i \la t,\bstildeX_j(\theta) \ra } \bigg|^2 \frac{1}{t^T \hat{\Gamma}_p t}  
 - \big|  \vp(t,\theta_0)  \\
 & & \quad\quad -  \vp(t,\theta)  \big|^2 \frac{1}{t^T \Gamma_p t}  \Bigg| w(t) \diff t 
\eeao
\begin{equation}\label{eq:Qcv33}
\begin{split}
 \leq & \int_{L_n^c} \Bigg| \bigg| \frac{1}{n} \sum_{j=1}^n e^{i \la t,\mathbf{X}_j \ra } - \frac{1}{H} \sum_{j=1}^H e^{i \la t,\bstildeX_j(\theta) \ra } \bigg|^2 - \big|  \vp(t,\theta_0)  -  \vp(t,\theta)  \big|^2 \Bigg| \frac{1}{t^T \hat{\Gamma}_p t} w(t) \diff t  \\
& +  4 \int_{L_n^c} \bigg| \frac{1}{t^T \hat{\Gamma}_p t} - \frac{1}{t^T \Gamma_p t} \Bigg|  w(t) \diff t.
\end{split}
\end{equation}
{By \ref{as:amixXt} and \ref{as:momX1u} it follows from Theorem~3(a) in Section~1.2.2 of \citet{doukhan94Mixing} that 
\begin{equation}\label{eq:covhzero}
|\cov(X_0,X_j)| \leq 8 \alpha_j^{\frac{1}{r}} 
\big( \E |X_1|^u \big)^{\frac{2}{u}} \rightarrow 0, \quad j \rightarrow \infty.    
\end{equation}
Since $\var(X_1) > 0$, it follows from \eqref{eq:covhzero} combined with Proposition~5.1.1 in \cite{Brockwell13} that $\det(\Gamma_p) > 0$, and therefore,} {the minimum eigenvalue $\lambda_{\min}(\Gamma_p)$ of $\Gamma_p$ is positive}. 
Thus, for all $t \in \R^p$,
\begin{equation}\label{eq:upqf}
\begin{split}
    t^T \Gamma_p t 
    & \ge 
    \lambda_{\min}(\Gamma_p)\, |t|^2  > 0.
\end{split}
\end{equation}
By \ref{as:ergodXt} and the ergodic theorem $\hat\Gamma_p\stas\Gamma_p$ and, since the eigenvalues of a matrix are continuous functions of its entries (cf. \citet{Bernstein05}, Fact 10.11.2), also  $\lambda_{\min}(\hat{\Gamma}_p) \stas \lambda_{\min}(\Gamma_p)>0$. 
It follows from \eqref{eq:upqf} and from the a.s. convergence of the eigenvalues that there exists $N > 0$ such that
\begin{equation}\label{eq:qfgnlb}
t^T \hat{\Gamma}_p t \geq |t|^2 \lambda_{\min}(\hat{\Gamma}_p) \geq |t|^2 \frac{\lambda_{\min}(\Gamma_p)}{2} > 0, \quad n \geq N.
\end{equation}
Thus, for $t\in L_n^c$ we obtain 
\begin{equation}\label{eq:btgtinv}
\bigg| \frac{1}{t^T \hat{\Gamma}_p t} - \frac{1}{t^T \Gamma_p t} \Bigg| 
\leq \frac{2}{k \lambda_{\min}(\Gamma_p)|t|^2} |t^T (\Gamma_p - \hat\Gamma_p)t|
\le \frac{2 |\Gamma_p - \hat\Gamma_p|}{k \lambda_{\min}(\Gamma_p)}. 
\end{equation}
This together with \eqref{eq:btgtinv} gives the following upper bound for the right-hand side of \eqref{eq:Qcv33}: 
\beam
& & \int_{\R^p} \Bigg| \bigg| \frac{1}{n} \sum_{j=1}^n e^{i \la t,\mathbf{X}_j \ra } - \frac{1}{H} \sum_{j=1}^H e^{i \la t,\bstildeX_j(\theta) \ra } \bigg|^2 - \big|  \vp(t,\theta_0)  -  \vp(t,\theta)  \big|^2 \Bigg|  \frac{w(t)}{k}  \diff t  \nonumber\\
& & \quad\quad +  \frac{{8} |\Gamma_p - \hat\Gamma_p| }{k \lambda_{\min}(\Gamma_p)} \int_{\R^p}  {w(t)}\diff t. \label{eq:Qcv2} 
\eeam
The first integral can be estimated as $|Q_{n,H}(\theta) - Q(\theta)|$ in \eqref{eq:con_2} which tends to 0 uniformly for $\theta\in\Theta$ provided that {\ref{as:intWt}} holds. Since 
$\hat\Gamma_p\stas\Gamma_p$, also the second integral in \eqref{eq:Qcv2} tends 0 a.s. as $n\to\infty$.

We turn to the integrated mean squared error $|Q^{\text{(cv)}}_{n,H,k}(\theta)-Q^{\text{(cv)}}(\theta)|$ on $L_n$. 
Let $L = \{t \in \R^p: |t| \le \sqrt{\frac{2k}{\lambda_{\min}(\Gamma_p)}}\}$. 
The control variates correction used in \eqref{eq:Qcv} can be regarded as a continuous function $g: \R^9 \mapsto \R^2$ whose entries are the arithmetic means defined in \eqref{eq:CVQ1}-\eqref{eq:CVQ4}. 
By \ref{as:EsupX4} and the uniform SLLN, 
each of these arithmetic means converge a.s. uniformly on $L \times \Theta$ as $\nto$ and $H \rightarrow \infty$. 
Thus, it follows from the continuity of $g$ and the continuous mapping theorem that
\begin{equation}\label{eq:supcvz}
\sup_{(t,\theta) \in L \times \Theta} |\kappa_{H}(t,\theta) |^2 \stas 0.
\end{equation}
For $n \geq N$ it follows from \eqref{eq:qfgnlb} that $L_n \subseteq L$ and thus using the \ inequality
\begin{equation*}
\begin{split}
\big| |a+b|^2 c - |d|^2e \big| & \leq \big| |a+b|^2 - |d|^2 \big||c| + |d|^2|c-e| \\
& \leq ( |a-d| + |b|)(4 + |b|)|c| + 4|c-e|,
\end{split}
\end{equation*}

valid for $a,b,c,d,e \in \C$ with $|d| \leq 2$ gives
\begin{equation*}
\begin{split}
& \int_{L_n} \Bigg| \bigg| \bigg( \frac{1}{n} \sum_{j=1}^n e^{i \la t,\mathbf{X}_j \ra } - \frac{1}{H} \sum_{j=1}^H e^{i \la t,\bstildeX_j(\theta) \ra }\bigg)  +  \kappa_{H}(t,\theta)  \bigg|^2 \frac{1}{t^T \hat{\Gamma}_p t}  \\
&  \quad\quad\quad - \big|  \vp(t,\theta_0)  -  \vp(t,\theta)  \big|^2 \frac{1}{t^T \Gamma_p t} \Bigg| w(t) \diff t \\
 \leq & \int_{L} \Bigg( \bigg| \frac{1}{n} \sum_{j=1}^n e^{i \la t,\mathbf{X}_j \ra } - 
 \vp(t,\theta_0) \bigg| + 
 \bigg| \frac{1}{H} \sum_{j=1}^H e^{i \la t,\bstildeX_j(\theta) \ra } - 
 \vp(t,\theta)  \bigg| \\
 & \quad\quad + |\kappa_{H}(t,\theta)| \bigg) \Big( 4 + |\kappa_{H}(t,\theta)| \Big) \frac{w(t) }{t^T \Gamma_p t} \diff t + 4 \int_{L} \bigg| \frac{1}{t^T \hat{\Gamma}_p t} - \frac{1}{t^T \Gamma_p t} \bigg| w(t) \diff t \\
 =: & I_{1,n}(\theta) + I_{2,n}(\theta).
\end{split}
\end{equation*}
From \eqref{eq:upqf}, \eqref{eq:supcvz}, \eqref{eq:ecfsup}, and \eqref{eq:ecfsup2} with $K_\delta=L$ for $\delta=\sqrt{2k/\lambda_{\min}(\Gamma_p)}$) ,and \ref{as:insuptd} it follows that $\suptheta I_{1,n}(\theta) \stas 0$ as $\nto$. 
Finally, \\ $\suptheta I_{2,n}(\theta) \stas 0$ by similar arguments as used in \eqref{eq:btgtinv} and \eqref{eq:Qcv2},
since for $t\in L$, also applying \ref{as:insuptd},
\begin{equation*}
\bigg| \frac{1}{t^T \hat{\Gamma}_p t} - \frac{1}{t^T \Gamma_p t} \Bigg| 
\leq \frac{2}{ (\lambda_{\min}(\Gamma_p))^2|t|^4} |t^T (\Gamma_p - \hat\Gamma_p)t|
\le \frac{2 |\Gamma_p - \hat\Gamma_p|}{ (\lambda_{\min}(\Gamma_p))^2|t|^2}\quad
\end{equation*}
and
$$
\int_{\R^p} \frac{w(t)}{|t|^2}\diff t<\infty.
$$

\noindent
\textbf{Proof of Theorem~\ref{th:asn}:}
By the definition of $\thehat$ in \eqref{eq:defThe} and under assumptions \ref{as:compact} and \ref{as:2diffXt} we have
\begin{equation*}
\gradtheta Q_{n,H}(\thehat)  = 0.
\end{equation*}
A Taylor expansion of order 1 of $\gradtheta Q_{n,H}$ around $\theta_0$ gives
$$
0 =  \gradtheta Q_{n,H}(\theta_0) + \gradtheta^2 Q_{n,H}(\theta_n)(\thehat - \theta_0)
$$
where $\theta_n \stas \theta_0$ as $\nto$. Therefore, asymptotic normality of $\sqrt{n}(\thehat - \theta_0)$ will follow by the delta method, if we prove that as $\nto$:
\begin{enumerate}
\item[$(1)$] $\sqrt{n}\gradtheta Q_{n,H}(\theta_0)$ converges weakly to a multivariate normal random variable, and 
\item[$(2)$] 
$\gradtheta^2 Q_{n,H}(\theta_n)$ converges in probability to a non-singular matrix.
\end{enumerate}

We start with the first point and compute the partial derivatives of $Q_{n,H}$:
\begin{equation}\label{eq:as2} 
\begin{split}
& \quad\quad \frac{\partial}{\partial{\theta^{(i)}}} Q_{n,H}(\theta) =  \frac{\partial}{\partial{\theta^{(i)}}} \bigg( \intrp   |  \vp_n(t) -   \vp_{H}(t,\theta)|^2 w(t) \diff t  \bigg) \\
& = \intrp \frac{\partial}{\partial{\theta^{(i)}}}   \Big( \Re ( \vp_n(t) -  \vp_{H}(t,\theta))^2 + \Im ( \vp_n(t) -  \vp_{H}(t,\theta))^2 \Big) w(t) \diff t \\
& = -2 \intrp     \Big( \Re ( \vp_n(t) -  \vp_{H}(t,\theta)) \frac{\partial}{\partial{\theta^{(i)}}} \Re ( \vp_{H}(t,\theta)) \\ 
& \quad\quad\quad\quad\quad + 
 \Im ( \vp_n(t) -  \vp_{H}(t,\theta)) \frac{\partial}{\partial{\theta^{(i)}}} \Im ( \vp_{H}(t,\theta))  \Big) w(t) \diff t, \quad i \in 1,\dots,q.
\end{split}
\end{equation}
Recall that  $\vp_n(t)$ and $\vp_H(t,\theta)$ denote the empirical characteristic functions of the observed blocks $(\bsX_1,\dots,\bsX_n)$ as in \eqref{eq:dfepcf} and of its Monte Carlo approximation $(\bstildeX_1(\theta),\dots,\bstildeX_H(\theta))$ as in \eqref{eq:ecfsim}, respectively. Define the partial derivatives of the real and imaginary part of $\vp_{H}(t,\theta)$:
\begin{equation}\label{eq:def_bni}
b_H^{(i)}(t,\theta) = \frac{1}{H} \sum_{j=1}^H 
\begin{pmatrix}
-\sin(\la t,\bstildeX_j(\theta) \ra)  \\
\cos(\la t,\bstildeX_j(\theta) \ra)
\end{pmatrix} \la t, \frac{\partial}{\partial{\theta^{(i)}}} \bstildeX_j(\theta) \ra , \quad i = 1,\dots,q,
\end{equation}
and summarize them into
\begin{equation}\label{eq:def_bn}
b_H(t,\theta) =
\begin{pmatrix}
(b_H^{(1)}(t,\theta))^T  \\ 
\vdots \\
(b_H^{(q)}(t,\theta))^T
\end{pmatrix}.
\end{equation}
Then consider
\begin{equation}\label{eq:def_fn}
\begin{pmatrix}
\Re ( \vp_n(t) - \vp(t,\theta_0) ) \\ 
\Im ( \vp_n(t) - \vp(t,\theta_0) )
\end{pmatrix} - \begin{pmatrix}
\Re ( \vp_H(t,\theta) - \vp(t,\theta_0) ) \\ 
\Im (\vp_H(t,\theta) - \vp(t,\theta_0) )
\end{pmatrix} =: g_n(t) -  \tilde{g}_H(t,\theta).
\end{equation}
Abbreviate $b_H(t):=b_H(t,\theta_0)$ and $\tilde g_H(t) := \tilde g_H(t,\theta_0)$.
Then it follows from  \eqref{eq:as2}, \eqref{eq:def_bn} and \eqref{eq:def_fn} that 
\begin{equation}\label{eq:GratQth}
\gradtheta Q_{n,H}(\theta_0) =  2 \intrp  b_H(t) g_n(t) w(t) \diff t -  2\intrp  b_H(t)\tilde{g}_H(t) w(t) \diff t.
\end{equation}
We analyze the asymptotic behavior of the first term in \eqref{eq:GratQth} in Lemma~\ref{le:FuncCon}. 
More precisely, we show there that $\int_{\kdelta} b_H(t) g_n(t) w(t) \diff t$ for $K_\delta$ as in \eqref{eq:def:kdel} converge in distribution to a $q$-dimensional Gaussian vector. 
Afterwards, Lemmas~\ref{le:FuncCon1} and \ref{le:FuncCon2} show that  as $\delta \rightarrow \infty$, componentwise in $\R^q$,
\begin{equation*}
\limsup_{\nto} \var \Big(  
 \int_{\kdelta^c} b_H(t) \sqrt{n}
 g_n(t)   w(t) \diff t
   \Big) \rightarrow 0,
 \end{equation*}
 and 
$$
\quad\mbox{and}\quad \int_{\kdelta^c}  \E[ b_1(t)] G(t) w(t) \diff t \stp 0
$$
where $G$ is a zero mean $\R^2$-valued Gaussian field. The formula given in \eqref{eq:GratQth} tells us that the term $\E[b_1(t,\theta)]$ will appear in the asymptotic covariance formula of the limiting distribution of the estimator. Therefore it is worth writing it in terms of the chf \eqref{eq:cf}.
\begin{remark}\label{re:rev1}
 For each $i \in \{1,\cdots,q\}$ and $\theta \in \Theta$, it follows from \eqref{eq:def_bni} that
 \begin{equation}\label{eq:rev1}
 \begin{split}
\E[b_1^{(i)}(t,\theta)]  & =  \E \begin{pmatrix}
-\sin(\la t,\bstildeX_j(\theta) \ra)  \\
\cos(\la t,\bstildeX_j(\theta) \ra)
\end{pmatrix} \la t, \frac{\partial}{\partial{\theta^{(i)}}} \bstildeX_j(\theta) \ra  \\ &  =  \E \Big( \frac{\partial}{\partial{\theta^{(i)}}} \cos(\la t,\bstildeX_j(\theta) \ra), \frac{\partial}{\partial{\theta^{(i)}}} \sin(\la t,\bstildeX_j(\theta) \ra)\Big)     
 \end{split}
 \end{equation}
Since both $\sin$ and $\cos$ are bounded by $1$ we can use \ref{as:momSD} to interchange expectation and differentiation in \eqref{eq:rev1}. This combined with \eqref{eq:def_bn} gives
 \begin{equation}\label{eq:rev2}
     \E[b_1(t,\theta)] = \Big( \frac{\partial}{\partial{\theta}}  \E \cos(\la t,\bstildeX_j(\theta) \ra),  \frac{\partial}{\partial{\theta}} \E \sin(\la t,\bstildeX_j(\theta) \ra)\Big) =\Big( \frac{\partial}{\partial{\theta}} \Re(\varphi(t,\theta)), \frac{\partial}{\partial{\theta}} \Im(\varphi(t,\theta)) \Big)
 \end{equation}
\end{remark}
This remark will be used later in the proof of Theorem~\ref{th:asn}.

We show by a standard Chebyshev argument that the second term in \eqref{eq:GratQth} converges in probability componentwise to 0 in \eqref{eq:cheb}. The convergence of the second derivatives $\gradtheta^2 Q_n(\theta_n)$ will be the topic of Lemma~\ref{le:ConvDeri}. For the scalar products above we use the following bounds several times below.

\begin{lemma}\label{lemma}
Let $\nu \geq 1$, $t \in \R^p$, $k,i \in \{1,\dots,q\}$ and $j \in \Z$ be fixed and assume that \ref{as:2diffXt} holds.Then the following bounds hold true.
\begin{itemize}
\item[(a)] If $\E | \gradtheta X_1(\theta)|^\nu < \infty$ for $\theta \in \Theta$, then there exists a constant $c > 0$ such that
\begin{equation}\label{eq:leeq1}
\E \Big| \la t, \diffthetak \bstildeX_j(\theta) \ra\Big|^\nu \leq 
c |t|^\nu \E | \gradtheta X_1(\theta)|^\nu,\quad t\in\R^p.
\end{equation}
\item[(b)] If $\E | \gradtheta^2 X_1(\theta)|^\nu < \infty$ for $\theta \in \Theta$, then there exists a constant $c > 0$ such that
\begin{equation}\label{eq:leeq2}
\E \Big| \la t, \diffthetaki \bstildeX_j(\theta) \ra \Big|^\nu \leq c |t|^\nu 
\E | \gradtheta^2 X_1(\theta)|^\nu,\quad t\in\R^p.
\end{equation}
\end{itemize}
The same bounds hold uniformly, taking expectations over $\suptheta$ or over $\sup_{t\in K}$  for some compact $K \subset \R^p$ at both sides of \eqref{eq:leeq1} and \eqref{eq:leeq2}, provided the corresponding expectations exist. 
\end{lemma}

\begin{proof}
(a) Applying the Cauchy-Schwarz inequality for the inner product, the fact that $(\bstildeX_j(\theta), \theta \in \Theta) \eqd (\bstildeX_1(\theta), \theta \in \Theta) \eqd (\bsX_1(\theta), \theta \in \Theta)$, bounding the $L^2$-norm by the $L^1$-norm, employing the inequality $| \sum_{j=1}^p \beta_j| ^{\nu} \leq p^{\nu-1}\sum_{j=1}^p |\beta_j| ^{\nu}$ valid for $\beta_1,\dots,\beta_p \in \R$ and $\nu \geq 1$ gives
\begin{equation}\label{eq:pro_le}
\begin{split}
& \E \Big| \la t, \diffthetak \bstildeX_j(\theta) \ra \Big|^\nu  \leq |t|^\nu \E \Big|  \diffthetak \bstildeX_j(\theta) \Big|^\nu  
= |t|^\nu \E \Big|  \diffthetak \bsX_1(\theta) \Big|^\nu \\
& \leq |t|^\nu  \E \bigg( \sum_{r=1}^p\Big|  \diffthetak X_r(\theta) \Big|\bigg)^\nu \leq p^{\nu-1}  |t|^\nu \sum_{r=1}^p \E \Big|  \diffthetak X_r(\theta) \Big|^\nu 
  \\
 & \leq  p^{\nu-1}  |t|^\nu \sum_{r=1}^p \E |  \gradtheta X_r(\theta) |^\nu = p^{\nu}  |t|^\nu \E |  \gradtheta X_1(\theta) |^\nu =: c|t|^\nu \E |  \gradtheta X_1(\theta) |^\nu.
\end{split}
\end{equation}
Part (b) follows by analogous calculations.
\end{proof}

\begin{lemma}\label{le:FuncCon}
Under assumptions \ref{as:ergodXt}, \ref{as:2diffXt}, \ref{as:amixXt}, \ref{as:momTigh} and \ref{as:mom1Der} we have on the Borel sets of $\R^q$,
\begin{equation}\label{eq:leFunC}
\int_{\kdelta} b_H(t) \sqrt{n}
   g_n(t)  w(t) \diff t \std \int_{\kdelta} \E[ b_1(t)] G(t)  w(t) \diff t, \quad\nto,
\end{equation}
where $G$ is an $\R^2$-valued Gaussian field.
\end{lemma}

\begin{proof}
Under assumptions \ref{as:amixXt} and \ref{as:momTigh}, it follows from Lemma~4.1(2) in \citet{davis18ADC} that $\sqrt{n}(\vp_n(\cdot) - \vp(\cdot,\theta_0))$ convergences in distribution on compact subsets of $\R^{p}$ to a complex-valued Gaussian field $\tilde G$, equivalently the vector of real and imaginary part converge to a bivariate Gaussian field $G$. 
Since the random elements $(\bstildeX_j(\theta), \theta \in \Theta)_{j \in \N}$ are iid and the partial derivatives exist by \ref{as:2diffXt}, also $(\bstildeX_j(\theta_0), \gradtheta \bstildeX_j(\theta_0) )_{j \in \N}$ are iid.
Then it follows from the definitions \eqref{eq:def_bni}, \eqref{eq:def_bn}, and Lemma~\ref{lemma} with $K = K_\delta)$ in combination with  \ref{as:mom1Der} that 
\beam\label{eq:FuncCon3}
\E \sup_{t \in K_{\delta}} |b_1(t)| \le c \sup_{t \in K_{\delta}} |t| \E |\gradtheta X_1(\theta_0)|\le c |\delta| \E |\gradtheta X_1(\theta_0)| <\infty.
\eeam
Hence, the uniform SLLN guarantees that 
\begin{equation*}
\sup_{t \in K_{\delta}} | b_H(t) - \E b_1(t) | \stas 0, \quad \nto.
\end{equation*}
Slutsky's theorem 
gives then $b_H(\cdot) \sqrt{n}g_n(\cdot,\theta_0)$ convergences in distribution on compact subsets of $\R^{p}$ to $\E[b_1(\cdot)] G(\cdot)$ as $\nto$. 
The result in \eqref{eq:leFunC} follows from the continuity of the integral by another application of the continuous mapping theorem on $C(K_{\delta})$.
\end{proof}

\begin{lemma}\label{le:FuncCon1}
Under assumptions \ref{as:2diffXt}, \ref{as:mom1Der} and \ref{as:inttWt} we have componentwise in $\R^q$,
\begin{equation}\label{eq:}
\limsup_{\nto} \var \bigg(  
 \int_{\kdelta^c} b_H(t) \sqrt{n}
 g_n(t)   w(t) \diff t
   \bigg) \rightarrow 0, \quad  \delta \rightarrow \infty.
\end{equation} 
\end{lemma}

\begin{proof}
Since $b_H(\cdot)$ and $g_n(\cdot)$ are independent and $\E g_n(t) = 0$, we have $\E[ b_H(t) g_n(t)] = 0$ for all $t\in\R^p$. 
An application of the Cauchy-Schwartz inequality for integrals gives
\beam
& & \quad\quad \var \bigg(  
 \int_{\kdelta^c}b_H(t) \sqrt{n}
 g_n(t)w(t) \diff t
   \bigg)  = \E \bigg(  
 \int_{\kdelta^c} 
b_H(t) \sqrt{n} g_n(t)   w(t) \diff t
   \bigg)^2 \nonumber\\
   & & \leq \bigg( \E \int_{\kdelta^c}|b_H(t)|^2  n  |g_n(t)|^2  w(t) \diff t \bigg) \bigg( \int_{\kdelta^c} w(t) \diff t \bigg) \label{eq:varIb}.
\eeam
We first obtain a bound for the product between the first component $g_{n,1}(\cdot)$ of $g_n(\cdot)$ and the first component $b_{H,1}^{(i)}(\cdot)$ of $b_H^{(i)}(\cdot)$.
Define for $t\in\R^p$
\beam
U_j(t) & = & \cos(\la t,\bsX_j \ra) - \Re(\vp(t,\theta_0)) \nonumber\\
 V_j(t) & = & -\sin(\la t,\bstildeX_j(\theta_0) \ra) \la t,\diffthetai \bstildeX_j(\theta_0) \ra,\quad j\in\Z\label{eq:def:UandV}.
\eeam
Then, 
$${g_{n,1}(t)}= \frac1n\sum_{j=1}^n U_j(t)\quad\mbox{and}\quad b_{H,1}^{(i)}(t) = \frac1H\sum_{j=1}^H V_j(t),\quad t\in\R^p.$$
Under \ref{as:amixXt} it follows from Theorem~3(a) in Section~1.2.2 of \citet{doukhan94Mixing} that for fixed $t$,
\begin{equation}\label{eq:ineqCov}
|\cov(U_0(t),U_j(t))| \leq 8 \alpha_j^{\frac{1}{r}} 
\big( \E |U_0(t)|^u \big)^{\frac{2}{u}},\quad j\in\N,
\end{equation}
where $u = \frac{2r}{(r-1)}$ and, thus, it follows from the stationarity of $(U_j(t))_{j \in \N}$ combined with \eqref{eq:ineqCov} and the fact that $|U_0(t)| \leq 2$  that 
\begin{equation}\label{eq:bound_fn}
\begin{split}
n \E \Big|\frac{1}{n} \sum_{j=1}^n U_j(t) \Big|^2 & =
\frac{1}{n}\sum_{j=1}^n \E U_j^2(t) + \frac2{n} \sum_{j=1}^{n-1} \Big(1 - \frac{k}{n}\Big) \E | U_0(t)U_j(t)|\\
& \leq \E U_0^2(t) + 16 \big( \E |U_0(t)|^u \big)^{\frac{2}{u}}  \sum_{j=1}^{\infty}  \alpha_j^{1/r} \\
& \leq  4 + 64 \sum_{j=1}^{\infty}  \alpha_j^{1/r} < \infty,
\end{split}
\end{equation}
where the bound is independent of $t$.
Recall that $H = H(n)=\bar{H} (n) n$. 
Under \ref{as:mom1Der}, it follows from the iid property of $(V_j(t))_{j \in \N}$
\begin{equation}\label{eq:bofn2}
\begin{split}
& n \E \Big|\frac{1}{H} \sum_{j=1}^H V_j(t) - \E V_0(t) \Big|^2  = n \var \bigg( \frac{1}{\bar{H} n} \sum_{j=1}^{\bar{H} n} V_j(t) \bigg) \\
& =  \frac{\E V_1^2(t)}{\bar{H} (n)}   \leq \frac{c |t|^2 \E |\gradtheta X_1(\theta_0)|^2}{\bar{H} (n)} \leq \frac{c |t|^2}{\bar{H} (n)}.
\end{split}
\end{equation}
Using the fact that $\Big|\frac{1}{n} \sum_{j=1}^n U_j (t)\Big| \leq 2$, adding and subtracting $\E V_0(t)$ with the inequality $|a+b|^2 \leq 2(|a|^2 + |b|^2)$, and \eqref{eq:bofn2} gives
\begin{equation}\label{eq:bofn3}
\begin{split}
 & n \E \Big|\frac{1}{n} \sum_{j=1}^n U_j(t) \Big|^2 \Big|\frac{1}{H} \sum_{j=1}^H V_j(t) \Big|^2 \\
& \leq 2 n \E \Big|\frac{1}{n} \sum_{j=1}^n U_j(t) \Big|^2  (\E V_0(t))^2 + 
8 n \E \Big|\frac{1}{H} \sum_{j=1}^H V_j(t) - \E V_0(t) \Big|^2 \\
& \leq c \Big(1+ \frac{|t|^2}{\bar{H} (n)}\Big).
\end{split}
\end{equation}
The calculations in \eqref{eq:bound_fn}, \eqref{eq:bofn2}, and \eqref{eq:bofn3}  can now be applied to show that for all $n \in \N$,
\begin{equation*}
n \E |g_n(t)|^2  |b_H(t)|^2 \leq c \Big(1+\frac{|t|^2}{\bar{H} (n)}\Big)
\end{equation*}
and, thus, it follows from \eqref{eq:varIb} together with \ref{as:intWt} and \ref{as:intt2Wt} that  
\begin{equation}\label{eq:limsupvz}
\begin{split}
& \limsup_{\nto} \var \bigg(  
 \int_{\kdelta^c}b_H(t) \sqrt{n}
 g_n(t)w(t) \diff t
   \bigg) \\
 &  \leq  \limsup_{\nto} \frac{c}{\bar{H} (n)}  \int_{\kdelta^c} (1+ |t|^2) w(t) \diff t  \int_{\kdelta^c}  w(t) \diff t \to 0, \quad \delta \rightarrow \infty.
 \end{split}
\end{equation}
\end{proof}

\begin{lemma}\label{le:FuncCon2}
Under assumptions \ref{as:2diffXt}, \ref{as:inttWt} and \ref{as:mom1Der}
\begin{equation*}
 \int_{\kdelta^c}  \E [b_1(t)] G(t) w(t) \diff t \stp 0, \quad \deltato.
\end{equation*}
\end{lemma}

\begin{proof}
It follows from \eqref{eq:def_bni}, \eqref{eq:def_bn}, \ref{as:mom1Der}, and \eqref{eq:FuncCon3}  
   \\$ \E |b_1(t)| 
    \leq c |t| \E  |\gradtheta X_1(\theta_0)| < \infty.$
Now we find an upper bound for the variance of each component of $G(t)$ for a fixed $t$. 
Let $U_j(t)$ be as defined at the left-hand side of \eqref{eq:def:UandV} and notice that the first component of $G(t)$ is the distributional limit of $\frac{1}{\sqrt{n}} \sum_{j=1}^n U_j(t)$. 
Since $(U_j(t))_{j\in\N}$ is $\alpha$-mixing by \ref{as:amixXt}, we can apply the CLT in \cite{Ibragimov71} (Theorem~18.5.3 with $\delta = 2/(r-1)$) and find that the variance of the first component of $G(t)$ is given by
\begin{equation*}
    \sigma^2_U = \E[ U_0^2(t)] + 2 \sum_{j=1}^\infty \E [U_0(t) U_j(t)].
\end{equation*}
This combined with Theorem~3(a) in Section~1.2.2 of \cite{doukhan94Mixing} and the fact that $\E U_j(t) = 0$ and $|U_j(t)| \leq 2$ for all $j \in \N$ gives by \ref{as:amixXt} and \eqref{eq:ineqCov}
\begin{equation*}
\begin{split}
| \sigma^2_U| & \leq 4 + \sum_{j=1}^ \infty |\cov (U_0(t), U_j(t))| \leq 4 + 8 \sum_{j=1}^ \infty  (2 \alpha_j)^{1/r} \big( \E |U_0(t)|^u \big)^{\frac{2}{u}} \leq  \\
& \quad\quad\quad 4 + 64 \sum_{j=1}^ \infty  (2 \alpha_j)^{1/r}.
\end{split}
\end{equation*}
A similar calculation shows that the variance of the second component of $G(t)$ is also bounded by a finite constant, which does not depend on $t$. 
Therefore, $\E|G(t)| \leq c$. 
This combined with \eqref{eq:FuncCon3} and assumption \ref{as:inttWt} gives
\begin{equation*}
\E \Big| \int_{\kdelta^c}  \E [b_1(t)] G(t) w(t) \diff t \Big| \leq c \E  |\gradtheta X_1(\theta_0)|   \int_{\kdelta^c} |t| w(t)  \diff t \rightarrow 0, \quad \deltato.
\end{equation*}
Since $L^1$-convergence implies convergence in probability the result follows.
\end{proof}
This proves part (1) of the delta method. We now turn to part (2). In order to calculate the second derivatives of $Q_{n,H}(\theta)$, which exist by  \ref{as:2diffXt}, we rewrite \eqref{eq:as2} as 
\beao
& & \frac{\partial}{\partial\theta^{(i)}} Q_{n,H}(\theta) \\ 
& & = - 2\int_{\R^d} \Big\{ \Big(\frac1n \sum_{j=1}^n \cos(\la t,\bsX_j(\theta)\ra ) -  \frac1H \sum_{j=1}^H \cos(\la t,\tilde \bsX_j(\theta)\ra ) \Big)
\frac{\partial}{\partial\theta^{(i)}} \Re(\vp_H(t,\theta))\\
& & +  \Big(\frac1n \sum_{j=1}^n \sin(\la t,\bsX_j(\theta)\ra ) -\frac1H \sum_{j=1}^H \sin(\la t,\tilde \bsX_j(\theta)\ra ) \Big)
\frac{\partial}{\partial\theta^{(i)}} \Im(\vp_H(t,\theta))\Big\} w(t) \diff t\\
& & =: 2 \int_{\R^d} \Big\{ i_{n,H}(t,\theta) j_{H,i}(t,\theta) - k_{n,H}(t,\theta) l_{H,i}(t,\theta)\Big\} w(t) \diff t.
\eeao
For the second derivatives we calculate for every $i,k\in\{1,\dots,q\}$,
\begin{equation}\label{eq:Qn2}
\begin{split}
\diffthetaki Q_{{n,H}}(\theta) & =  {2} \intrp \Big\{ {j_{H,k}}(t,\theta) {j_{H,i}}(t,\theta) + {i_{n,H}}(t,\theta) {g_{H,k,i}}(t,\theta) \\
& + {l_{H,k}}(t,\theta) {l_{H,i}}(t,\theta) - {k_{n,H}}(t,\theta) {h_{H,k,i}}(t,\theta) \Big\} w(t) \diff t,
\end{split}
\end{equation}
where we summarize all quantities used in the following list:
\beao
 {i_{n,H}}(t,\theta) & = & \frac{1}{n} \sum_{j=1}^n  \cos(\la t,\bsX_j \ra) -  \frac{1}{H} \sum_{j=1}^H \cos(\la t,\bstildeX_j(\theta) \ra) \label{eq:defin} \\
{j_{H,i}}(t,\theta)  & = &   \diffthetai {i_{n,H}}(t,\theta) =\frac{1}{H} \sum_{j=1}^H \sin(\la t,\bstildeX_j(\theta) \ra) \la t, \diffthetai \bstildeX_j(\theta) \ra \label{eq:defjn} \\
{g_{H,k,i}}(t,\theta) & = & \diffthetak {j_{H,i}}(t,\theta) \\
& = & \frac{1}{H} \sum_{j=1}^H \cos(\la t,\bstildeX_j(\theta) \ra) \la t, \diffthetak \bstildeX_j(\theta) \ra  \la t, \diffthetai \bstildeX_j(\theta) \ra  \nonumber\\
& &    \quad\quad\quad\quad + \sin(\la t,\bstildeX_j(\theta) \ra) \la t, \diffthetaki \bstildeX_j(\theta) \ra \label{eq:defgn} 
\eeao
\beao
{k_{n,H}}(t,\theta)  & = & \frac{1}{n} \sum_{j=1}^n \sin(\la t,\bsX_j \ra) -  \frac{1}{H} \sum_{j=1}^H \sin(\la t,\bstildeX_j(\theta) \ra)\label{eq:defkn} \\
{l_{H,i}}(t,\theta) & = &   - \diffthetai {k_{n,H}}(t,\theta) = \frac{1}{H} \sum_{j=1}^H \cos(\la t,\bstildeX_j(\theta) \ra) \la t, \diffthetai \bstildeX_j(\theta) \ra \label{eq:defln} \\
{h_{H,k,i}}(t,\theta) & = & \diffthetak {l_{H,i}}(t,\theta) \nonumber\\
& = & \frac{1}{H} \sum_{j=1}^H -\sin(\la t,\bstildeX_j(\theta) \ra) \la t, \frac{\partial}{\partial{\theta^{(k)}}} \bstildeX_j(\theta) \ra   \la t, \diffthetai \bstildeX_j(\theta) \ra  \nonumber\\
 &  &   \quad\quad\quad\quad + \cos(\la t,\bstildeX_j(\theta) \ra) \la t, \diffthetaki \bstildeX_j(\theta) \ra. \label{eq:defhn}
\eeao

\begin{lemma}\label{le:ConvDeri}
If the assumptions \ref{as:ergodXt}, \ref{as:contXt}, \ref{as:2diffXt}, \ref{as:momSD}, \ref{as:intt2Wt} hold and $(\theta_n)_{n\in\N}\subset\Theta$ satisfying 
$\theta_n \stas \theta_0$, then for every $k,i \in \{1,\dots,q\}$, {as $\nto$}
\begin{equation}\label{eq:diffThCov}
\begin{split}
 & \diffthetaki Q_{n,H}(\theta_n) 
  \\
 &  \stp \intrp \Big( \E j_{1,k}(t,\theta_0) \E j_{1,i}(t,\theta_0) + \E l_{1,k}(t,\theta_0) \E l_{1,i}(t,\theta_0) \Big) w(t) \diff t.
\end{split}
\end{equation}
\end{lemma}

\begin{proof}
We first prove {that as $\nto$}
\begin{equation}\label{eq:ingnint}
\intrp  {i_{n,H}}(t,{\theta_n}) {g_{H,k,i}}(t,{\theta_n}) w(t) \diff t \, \stp \, \intrp \E{i_{1,1}}(\theta_0,t)\E g_{1,k,i}(\theta_0,t) w(t) \diff t.
\end{equation}
\textbf{Step 1: Uniform convergence on $\Theta$:} It follows from the iid property of the random elements  $(\tilde\bsX_j(\theta), \theta \in \Theta)_{j \in \N}$ 
that the sequence \\
$( \tilde\bsX_j(\theta), \gradtheta  \tilde\bsX_j(\theta),  \gradtheta^2  \tilde\bsX_j(\theta), \theta \in \Theta)_{j \in \N}$ is iid. 
Lemma~\ref{lemma} together with \ref{as:momSD} gives the uniform bound
\begin{equation*}
\E \suptheta |g_{1,k,i}(t,\theta)| \leq c \Big( |t|^2 \E \suptheta |\gradtheta X_1(\theta)|^2 + |t| \E \suptheta |\gradtheta^2 X_1(\theta)| \Big)  < \infty,
\end{equation*}
and it follows from the uniform SLLN 
that for every fixed $t \in \R^p$
\begin{equation}\label{eq:convUnif_ingn}
\suptheta |{g_{H,k,i}}(t,\theta) - \E g_{1,k,i}(t,\theta)| \stas 0, \quad \nto.
\end{equation}
Similarly,
\begin{equation}\label{eq:convUnif_i_a}
\suptheta \bigg|  \frac{1}{H} \sum_{j=1}^H \cos(\la t,\bstildeX_j(\theta) \ra) - \Re(\vp(t,\theta)) \bigg|  \stas 0, \quad \nto.
\end{equation}
Because of \ref{as:ergodXt} the ergodic theorem gives
\begin{equation}\label{eq:conv_i_b}
\frac{1}{n} \sum_{j=1}^n \cos(\la t,\bsX_j \ra) \stas \Re(\vp(t,\theta_0)), \quad \nto.
\end{equation}
Therefore, \eqref{eq:convUnif_i_a} combined with \eqref{eq:conv_i_b} and the triangle inequality imply
\begin{equation}\label{eq:convUnif_i}
\suptheta |i_{n,H}(t,\theta) - \E i_{1,1}(t,\theta)| \stas 0, \quad \nto.
\end{equation}
\noindent
\textbf{Step 2: Pointwise convergence of ${i_{n,H}}(t,\theta_n) {g_{H,k,i}}(t,\theta_n)$:}
The triangle inequality implies  
\begin{equation}\label{eq:boundGadd}
\begin{split}
& |{i_{n,H}}(t,\theta_n){g_{H,k,i}}(t,\theta_n) - \E {i_{1,1}}(\theta_0,t)\E g_{1,k,i}(\theta_0,t)| \\ 
 \leq & | {i_{n,H}}(t,\theta_n){g_{H,k,i}}(t,\theta_n) - \E {i_{1,1}}(t,\theta_n)\E g_{1,k,i}(t,\theta_n)| \\ 
& + |\E {i_{1,1}}(t,\theta_n)\E g_{1,k,i}(t,\theta_n) - \E {i_{1,1}}(\theta_0,t)\E g_{1,k,i}(\theta_0,t)| \\
 \leq & \suptheta \big\{ |{i_{n,H}}(t,\theta){g_{H,k,i}}(t,\theta) - \E {i_{1,1}}(t,\theta)\E g_{1,k,i}(t,\theta)|\big\} \\
& + |\E {i_{1,1}}(t,\theta_n)\E g_{1,k,i}(t,\theta_n) - \E {i_{1,1}}(\theta_0,t)\E g_{1,k,i}(\theta_0,t)|. 
\end{split}
\end{equation}
Since $\theta_n \stas \theta_0$ and the map $\theta \mapsto \E{i_{1,1}}(t,\theta)\E g_{1,k,i}(t,\theta)$ is continuous in $\Theta$, (by \ref{as:2diffXt} and \ref{as:momSD}) it follows that the second term {on the right-hand side of} \eqref{eq:boundGadd} converges a.s. to zero. 
Additionally, since the uniform convergences on \eqref{eq:convUnif_ingn} {and \eqref{eq:convUnif_i} imply} the uniform convergence of the product ${i_{n,H}}(t,\theta){g_{H,k,i}}(t,\theta)$ on $\Theta$ it follows that the first term on the right-hand side of \eqref{eq:boundGadd} also converges a.s. to zero.\\[2mm]
\textbf{Step 3: $L^1$-convergence}: Since we have already shown a.s. convergence, it follows from Theorems~6.25(iii) and 6.19  in \cite{Klenke13Prob} (with $H(x) = |x|^{1+\varepsilon}$) that $L^1$-convergence follows provided that
\begin{equation*}
\sup_{n \in \N} \E |{i_{n,H}}(t,\theta_n) {g_{H,k,i}}(t,\theta_n)|^{1+\eps} < \infty  
\end{equation*}
for some $\eps > 0$. 
Using the fact that $|{i_{n,H}}(t,\theta_n)| \leq 2$ and the inequality $|\frac{1}{n} \sum_{j=1}^n {\beta}_j| ^{1+\eps} \leq \frac{1}{n} \sum_{j=1}^n |{\beta}_j| ^{1+\eps}$, ${\beta_1,\dots,\beta_n \in \R}$, 
we obtain
\begin{equation}\label{eq:ingnE1ep}
\begin{split}
 &    \E |{i_{n,H}}(t,\theta_n) {g_{H,k,i}}(t,\theta_n)|^{1+\eps} \\
  & \leq 2^{1+\eps}  \E |{g_{H,k,i}}(t,\theta_n)|^{1+\eps} \\
  & \leq {\frac{2^{1+\eps}}{H}} \sum_{j=1}^{{H}} \E \Big|\cos(\la t,\bstildeX_j(\theta_n) \ra) \la t, \diffthetak \bstildeX_j(\theta_n) \ra  \la t, \diffthetai \bstildeX_j(\theta_n) \ra  \\
  & \quad\quad +  \sin(\la t,\bstildeX_j(\theta_n) \ra) \la t, \diffthetaki \bstildeX_j(\theta_n) \ra\Big|^{1 + \eps}\\
  & \leq  {\frac{2^{1+\eps}}{H}} \sum_{j=1}^{{H}} \E \Big| \la t, \diffthetak \bstildeX_j(\theta_n) \ra  \la t, \diffthetai \bstildeX_j(\theta_n) \ra   +  \la t, \diffthetaki \bstildeX_j(\theta_n) \ra\Big|^{1 + \eps},
  \end{split}
\end{equation}
since $|\cos(\cdot)|, |\sin(\cdot)| \leq 1$. 
Now we use the inequality $|a+b|^{1+\eps}\le 2^\eps(|a|^{1+\eps}+|b|^{1+\eps})$ for $a,b\in\R$, assumption \ref{as:momSD} for the uniform bound in Lemma~\ref{lemma} and the fact that the sequence $(\tilde\bsX_j(\theta), \gradtheta  \tilde\bsX_j(\theta),  \gradtheta^2  \tilde\bsX_j(\theta), \theta \in \Theta)_{j \in \N}$ is iid to continue
\begin{equation}\label{eq:ingnE1ep}
\begin{split}
  & \leq 2^{1+2\eps} \frac{1}{H} \sum_{j=1}^H \bigg( \E \Big| \la t, \diffthetak \bstildeX_j(\theta_n) \ra  \la t, \diffthetai \bstildeX_j(\theta_n) \ra\Big|^{1 + \eps}
  \\
  & + \E \Big| \la t, \diffthetaki \bstildeX_j(\theta_n) \ra\Big|^{1 + \eps} \bigg) \\
  & \leq c \frac{1}{H} \sum_{j=1}^H \big( |t|^{2(1+\eps)} \E | \gradtheta X_1(\theta_n)|^{2(1+\eps)}  + |t|^{1+\eps} \E  | \gradtheta^2 X_1(\theta_n)|^{1+\eps} \big) \\
  & \leq c \Big( |t|^{2(1+\eps)} \E \suptheta | \gradtheta X_1(\theta)|^{2(1+\eps)} +  |t|^{1+\eps} \E \suptheta | \gradtheta^2 X_1(\theta)|^{1+\eps} \Big) := v(t) < \infty.
\end{split}
\end{equation}
\textbf{Step 4: Convergence of the random integrals:} Define the sequence of functions
\begin{equation*}
v_n(t) = \E |{i_{n,H}}(t,\theta_n) {g_{H,k,i}}(t,\theta_n) - \E{i_{1,1}}(\theta_0,t)\E g_{1,k,i}(\theta_0,t)|, \quad t \in \R^p,
\end{equation*}
and recall that from the $L^1$-convergence showed in Step 3, for every $t \in \R^p$ we have $v_n(t) \rightarrow 0$ as $\nto$. From the definition of the function $v$ in the last line of \eqref{eq:ingnE1ep} it follows that $\sup_{n \in \N} v_n(t) \leq 2 v(t)$. Additionally, assumption \ref{as:intt2Wt} implies that
\begin{equation*}
 \intrp v(t) w(t) \diff t < \infty.
\end{equation*}
Therefore, it follows from Fubini's Theorem and dominated convergence that
\beam
 &  &   \E \bigg| \intrp \big( {i_{n,H}}(t,\theta_n){g_{H,k,i}}(t,\theta_n) - \E{i_{1,1}}(\theta_0,t)\E g_{1,k,i}(\theta_0,t) \big) w(t) \diff t \bigg| \nonumber\\
& & \leq  \E \intrp  | {i_{n,H}}(t,\theta_n){g_{H,k,i}}(t,\theta_n) - \E{i_{1,1}}(\theta_0,t)\E g_{1,k,i}(\theta_0,t) | w(t) \diff t  \nonumber\\
& & = \intrp v_n(t) w(t) \diff t \rightarrow 0,\quad \nto \label{eq:conv_integralG},
\eeam
and therefore the convergence in probability of \eqref{eq:ingnint} follows from the $L^1$-convergence in \eqref{eq:conv_integralG}. 

The proofs for the other three remaining integrals on the right-hand side of \eqref{eq:Qn2} follow along the same lines. 
The result in \eqref{eq:diffThCov} is then a consequence of the fact that for all $t \in \R^p$, $\E i_{1,1}(t,\theta_0) = \E k_{1,1}(t,\theta_0) = 0$.
\end{proof}

\noindent
\textbf{Proof of Theorem~\ref{th:asn}}: We handle each term in \eqref{eq:GratQth} separately. As a direct consequence of Theorem~\ref{th:cons} and Lemmas~\ref{le:FuncCon}-\ref{le:ConvDeri}, 
$$
-2 {(\gradtheta^2 Q_{n,H}(\theta_n))^{-1}} \intrp  b_H(t) {\sqrt{n}}g_n(t) w(t) \diff t \std N(0,Q^{-1} W Q^{-1}), \quad \nto,
$$
where $Q = (Q_{k,i})_{k,i=1}^q$ with 
\begin{equation}\label{eq:Qki}
Q_{k,i} = \intrp \Big( \E j_{1,k}(t,\theta_0) \E j_{1,i}(t,\theta_0) + \E l_{1,k}(t,\theta_0) \E l_{1,i}(t,\theta_0) \Big) w(t) \diff t,
\end{equation}
\begin{equation*}
W = \var \bigg( \int_{\R^p} \E[b_1(t)] G(t)  w(t) \diff t\bigg),
\end{equation*}
and $G$ being the $\R^2$-valued Gaussian field from Lemma~\ref{le:FuncCon}. For arbitrary $k,r \in \{1,\dots,q\}$ we have
\beam
W_{k,r} & = & \cov \bigg( \int_{\R^p} \E [b^{(k)}_1(t)]^T G(t) w(t) \diff t, \int_{\R^p} \E [b^{(r)}_1(t)]^T G(t) w(t) \diff t \bigg)  \nonumber \\
& = & \int_{\R^{p}} \int_{\R^{p}} \E [b^{(k)}_1(t)]^T \E [G(t) G(s)^T ]  \E [b^{(k)}_1(s)] w(t) w(s) \diff t \diff s \label{eq:Wkr}.
\eeam
Since $(X_j)_{j \in \N}$ is $\alpha$-mixing by \ref{as:amixXt}, we can apply the CLT in \cite{Ibragimov71} (Theorem~18.5.3 with $\delta = 2/(r-1)$) and find that
\begin{equation}\label{eq:Egtgs}
\E [G(t) G(s)^T] = \E [F_1(t) F_1(s)^T] + 2 \sum_{j=2}^\infty \E [ F_1(t) F_j(s)^T],
\end{equation}
where
\begin{equation}\label{eq:fjt}
 F_j(t) = 
 \begin{pmatrix}
\cos(\la t, \bsX_j \ra) - \Re(\vp(t,\theta_0)) \\ 
  \sin(\la t, \bsX_j \ra) - \Im(\vp(t,\theta_0))
\end{pmatrix}.
\end{equation}
Substituting \eqref{eq:Egtgs} and \eqref{eq:fjt} into \eqref{eq:Wkr} gives with Fubini's Theorem 
\begin{equation*}
\begin{split}
W_{k,r} & = \int_{\R^{p}} \int_{\R^{p}} \bigg\{ \E[ b^{(k)}_1(t)]^T \bigg( \E [F_1(t) F_1(s)^T] + 2 \sum_{j=2}^\infty \E  [F_1(t) F_j(s)^T]  \bigg)  \times \\
& \quad\quad\quad\quad \E [b^{(k)}_1(s)]  w(t) w(s) \bigg\}  \diff t \diff s \\
& = \int_{\R^{p}} \int_{\R^{p}} \E [b^{(k)}_1(t)]^T \E [F_1(t) F_1(s)^T] \E [b^{(k)}_1(s)]  w(t) w(s)  \diff t \diff s \\ 
& \quad\quad + 2 \sum_{j=2}^\infty \int_{\R^{p^2}} \E [b^{(k)}_1(t)]^T \E  [F_1(t) F_j(s)^T]  \E [b^{(k)}_1(s)]  w(t) w(s)  \diff t \diff s \\
& = \E \Big( \int_{\R^p} \E [b^{(k)}_1(t)]^T F_1(t) w(t) \diff t \Big)^2 \\
& \quad\quad +   2 \sum_{j=2}^\infty \E \bigg[ \Big( \int_{\R^p} \E [b^{(k)}_1(t)]^T F_1(t) w(t) \diff t \Big)   \\
& \quad\quad\quad\quad\quad\quad\quad\quad \times \Big( \int_{\R^p} \E [b^{(k)}_1(s)]^T F_j(t) w(s) \diff s \Big) \bigg],
\end{split}
\end{equation*}
which combined with Remark~\ref{re:rev1} gives \eqref{eq:Wdeffor}. By the same arguments of interchanging expectation and differentiation from Remark~\ref{re:rev1} we obtain
$$
\E j_{1,i}(t,\theta) = - \diffthetai \Re (\varphi(t,\theta)) \quad \text{and} \quad \E l_{1,i}(t,\theta) = \diffthetai \Im (\varphi(t,\theta)). 
$$
This together with \eqref{eq:Qki} gives
$$
Q_{k,i} = \intrp \Big( \diffthetak \Re (\varphi(t,\theta_0)) , \diffthetak  \Im (\varphi(t,\theta_0)) \Big)  \Big( \diffthetai \Re (\varphi(t,\theta_0)) , \diffthetai  \Im (\varphi(t,\theta_0)) \Big)^T  w(t) \diff t,
$$
leading to \eqref{eq:rev3}.

The second term in \eqref{eq:GratQth} is, up to a constant, 
\begin{equation*}
\intrp  b_H(t)\tilde{g}_H(t) w(t) \diff t.
\end{equation*}
It follows from the fact that $(\bstildeX_j(\theta_0))_{j \in \N} \eqd (\bsX_j)_{j \in \N}$ combined with \eqref{eq:limsupvz} that
\begin{equation}\label{eq:cheb}
\begin{split}
&  \var \bigg(  \intrp  b_H(t)\sqrt{n} \tilde{g}_H(t) w(t) \diff t \bigg) \\
&  \leq \frac{c}{\bar{H} (n)} \bigg(  \intrp (1 +  |t|^2) w(t) \diff t \bigg)\bigg(  \intrp w(t) \diff t \bigg) =: \frac{c}{\bar{H} (n)} \rightarrow 0,
\end{split}
\end{equation}
as $\nto$. Thus \eqref{eq:anthe} follows from Chebyshev's inequality.


\subsection{Finite sample behavior of the estimators}\label{se:tables}

\subsubsection{{ARFIMA models driven by noise from Gaussian, Laplace, and Student-\texorpdfstring{$t$}{Lg} distributions}}

\begin{table}[H]
\centerline{{ARFIMA model driven by standard Gaussian noise} }
\centering
\def\arraystretch{0.8}
\begin{tabular}{lrrcrrcrrc}
  \hline
  & \multicolumn{3}{c}{$d=0.05$} & \multicolumn{3}{c}{$d=0.10$} & \multicolumn{3}{c}{$d=0.15$} \\ 
  & Bias & Std & RMSE & Bias & Std & RMSE & Bias & Std & RMSE \\ 
  $\thehat$ & 0.000 & 0.056 & 0.056 & 0.002 & 0.054 & 0.054 & 0.004 & 0.049 & 0.049 \\ 
  $\hat\theta_n$ & -0.005 & 0.050 & 0.050 & -0.004 & 0.047 & 0.047 & -0.004 & 0.044 & 0.045 \\ 
  MLE & -0.015 & 0.040 & {0.043} & -0.015 & 0.040 & {0.043} & -0.016 & 0.040 & {0.043} \\   \\
& \multicolumn{3}{c}{$d=0.20$} & \multicolumn{3}{c}{$d=0.25$} & \multicolumn{3}{c}{$d=0.30$} \\   
  & Bias & Std & RMSE & Bias & Std & RMSE & Bias & Std & RMSE \\ 
 $\thehat$ & 0.003 & 0.047 & 0.047 & 0.000 & 0.046 & 0.046 & -0.003 & 0.048 & 0.048 \\ 
$\hat\theta_n$ & -0.004 & 0.045 & 0.045 & -0.006 & 0.044 & 0.044 & -0.007 & 0.046 & 0.047 \\ 
  MLE & -0.016 & 0.040 & {0.043} & -0.017 & 0.039 & {0.043} & -0.017 & 0.039 & {0.043} \\  \\
& \multicolumn{3}{c}{$d=0.35$} & \multicolumn{3}{c}{$d=0.40$} & \multicolumn{3}{c}{$d=0.45$} \\   
  & Bias & Std & RMSE & Bias & Std & RMSE & Bias & Std & RMSE \\ 
   $\thehat$ & -0.006 & 0.050 & 0.051 & -0.013 & 0.051 & 0.052 & -0.022 & 0.047 & 0.052 \\ 
   $\hat\theta_n$ & -0.009 & 0.049 & 0.050 & -0.013 & 0.051 & 0.052 & -0.021 & 0.048 & 0.052 \\ 
  MLE & -0.019 & 0.039 & {0.043} & -0.021 & 0.037 & {0.043} & -0.027 & 0.034 & {0.043} \\ 
   \hline
\end{tabular}
\caption{Comparison of the simulation based estimator $\thehat$ for $H=3\,000$, the oracle estimator $\hat\theta_n$, and the MLE for sample size $n = 400$. 
For all estimators we have taken $p = 3$ with $w$ the Gaussian density as in \eqref{eq:wtGau}.
Reported results are based on 500 replications. 
}\label{tb:arfimaGa}
\end{table}

\newpage

\begin{table}[H]
\centerline{{ARFIMA model driven by standard Laplace noise} }
\centering
\begin{tabular}{lrrcrrcrrc}
  \hline
    & \multicolumn{3}{c}{$d=0.05$} & \multicolumn{3}{c}{$d=0.10$} & \multicolumn{3}{c}{$d=0.15$} \\ 
 & Bias & Std & RMSE & Bias & Std & RMSE & Bias & Std & RMSE \\ 
  $\thehat$ & -0.004 & 0.062 & 0.062 & -0.003 & 0.060 & 0.060 & 0.004 & 0.054 & 0.054 \\ 
  $\hat\theta_n$ & -0.005 & 0.051 & 0.051 & -0.003 & 0.049 & 0.049 & 0.001 & 0.048 & 0.047 \\ 
  {QMLE}  & -0.012 & 0.043 & {0.045} & -0.013 & 0.043 & {0.045} & -0.013 & 0.043 & {0.045} \\  \\
& \multicolumn{3}{c}{$d=0.20$} & \multicolumn{3}{c}{$d=0.25$} & \multicolumn{3}{c}{$d=0.30$} \\   
  $\thehat$ & 0.008 & 0.049 & 0.050 & 0.012 & 0.051 & 0.053 & 0.012 & 0.049 & 0.051 \\ 
  $\hat\theta_n$ & 0.005 & 0.047 & 0.047 & 0.009 & 0.046 & 0.047 & 0.010 & 0.046 & 0.047 \\ 
   QMLE & -0.014 & 0.042 & {0.044} & -0.014 & 0.042 & {0.044} & -0.015 & 0.042 & {0.044} \\  \\
& \multicolumn{3}{c}{$d=0.35$} & \multicolumn{3}{c}{$d=0.40$} & \multicolumn{3}{c}{$d=0.45$} \\   
  $\thehat$ & 0.009 & 0.045 & 0.046 & -0.004 & 0.042 & 0.042 & -0.022 & 0.037 & 0.043 \\ 
  $\hat\theta_n$ & 0.006 & 0.044 & {0.044} & -0.004 & 0.040 & {0.040} & -0.023 & 0.035 & {0.042}  \\ 
  QMLE & -0.016 & 0.041 & {0.044} & -0.019 & 0.039 & 0.044 & -0.025 & 0.035 & 0.043 \\ 
   \hline
\end{tabular}
\caption{Comparison of the simulation based estimator $\thehat$ for $H=3\,000$, {the quasi-oracle estimator $\hat\theta_n$,} and the QMLE 
for sample size $n = 400$.
For all estimators we have taken $p = 3$ with $w$ the Gaussian density as in \eqref{eq:wtGau}.
Reported are results based on 500 replications. 
}\label{tb:arfimadExp}
\end{table}

\begin{table}[H]
\centerline{{ARFIMA model driven by standard Student-$t$ noise with 6 degrees of freedom} }
\centering
\begin{tabular}{lrrcrrcrrc}
  \hline
 & Bias & Std & RMSE & Bias & Std & RMSE & Bias & Std & RMSE \\ 
    & \multicolumn{3}{c}{$d=0.05$} & \multicolumn{3}{c}{$d=0.10$} & \multicolumn{3}{c}{$d=0.15$} \\ 
  $\thehat$ & -0.002 & 0.063 & 0.063 & 0.005 & 0.059 & 0.060 & 0.008 & 0.053 & 0.054 \\ 
  $\hat\theta_n$ & -0.002 & 0.052 & 0.052 & -0.001 & 0.050 & 0.050 & 0.000 & 0.048 & 0.048 \\ 
  {QMLE} & -0.012 & 0.039 & {0.041} & -0.012 & 0.039 & {0.041} & -0.013 & 0.039 & {0.041} \\ \\
& \multicolumn{3}{c}{$d=0.20$} & \multicolumn{3}{c}{$d=0.25$} & \multicolumn{3}{c}{$d=0.30$} \\   
   $\thehat$ & 0.008 & 0.049 & 0.050 & 0.004 & 0.050 & 0.050 & 0.008 & 0.050 & 0.050 \\ 
  $\hat\theta_n$ & 0.001 & 0.047 & 0.047 & 0.002 & 0.047 & 0.047 & 0.001 & 0.047 & 0.047 \\ 
  QMLE & -0.013 & 0.039 & {0.041} & -0.014 & 0.039 & {0.041} & -0.014 & 0.039 & {0.041} \\  \\
& \multicolumn{3}{c}{$d=0.35$} & \multicolumn{3}{c}{$d=0.40$} & \multicolumn{3}{c}{$d=0.45$} \\   
   $\thehat$ & 0.002 & 0.049 & 0.049 & -0.009 & 0.043 & 0.044 & -0.029 & 0.038 & 0.048 \\ 
  $\hat\theta_n$ & -0.004 & 0.046 & 0.046 & -0.014 & 0.043 & 0.045 & -0.031 & 0.038 & 0.049 \\ 
  QMLE & -0.016 & 0.038 & {0.041} & -0.018 & 0.037 & {0.041} & -0.024 & 0.033 & {0.041} \\ 
   \hline
\end{tabular}
\caption{Comparison of the simulation based estimator $\thehat$ for $H=3\,000$, {the quasi-oracle estimator $\hat\theta_n$,} and the QMLE 
for sample size $n = 400$. 
For all estimators we have taken $p = 3$ with $w$ the Gaussian density as in \eqref{eq:wtGau}.
Reported results are based on 500 replications.
}\label{tb:arfimadStudent}
\end{table}

\newpage

\subsubsection{Poisson-AR model}

\begin{table}[ht]
\centerline{{Poisson-AR(1) model} }
\centering
\begin{small}
\setlength{\tabcolsep}{5pt}
\def\arraystretch{0.7}
\begin{tabular}{l|ccccccccccc}
  \hline
 & $\beta$ & $\phi$ & $\sigma$ &$\beta$ & $\phi$ & $\sigma$ & $\beta$ & $\phi$ & $\sigma$ \\ 
  \hline
  \multicolumn{10}{c}{$D=10$} \\
TRUE & -0.613 & -0.500 & 1.236 & -0.613 & 0.500 & 1.236 & -0.613 & 0.900 & 0.622 \\ 
 Bias($\thehat$) & -0.015 & 0.025 & 0.002 & -0.012 & 0.014 & -0.032 & -0.016 & -0.010 & 0.002 \\ 
  RMSE($\thehat$) & \colorbox{lg}{0.096} & \colorbox{lg}{0.101} & \colorbox{lg}{0.119} & 0.148 & 0.107 & 0.120 & 0.298 & 0.054 & \colorbox{lg}{0.128} \\
  Bias($\thehatcv$) & 0.023 & 0.031 & -0.007 & 0.006 & 0.002 & -0.018 & 0.061 & -0.007 & -0.036 \\ 
  RMSE($\thehatcv$) &  0.102 & 0.129  & 0.122  & \colorbox{lg}{0.138} &  \colorbox{lg}{0.098} &  \colorbox{lg}{0.098} & \colorbox{lg}{0.285} &  \colorbox{lg}{0.049} &  0.132 \\  
  \multicolumn{10}{c}{$D=1$} \\
  TRUE  & 0.150 & -0.500 & 0.619 & 0.150 & 0.500 & 0.619 & 0.150 & 0.900 & 0.312 \\ 
  Bias($\thehat$) & -0.004 & 0.024 & -0.016 & -0.006 & 0.005 & -0.023 & -0.016 & -0.033 & 0.028 \\ 
  RMSE($\thehat$)  & 0.057 & 0.144 & 0.088 & 0.074 & 0.141 & 0.081 & 0.147 & 0.084 & 0.095 \\
  Bias($\thehatcv$) & 0.003 & -0.011 & -0.017 & 0.001 & 0.023 & -0.019 & 0.003 & -0.009 & -0.012 \\ 
  RMSE($\thehatcv$)  & \colorbox{lg}{0.055} &  \colorbox{lg}{0.124} &  \colorbox{lg}{0.085} & \colorbox{lg}{0.071} & \colorbox{lg}{0.102} &  \colorbox{lg}{0.069} & \colorbox{lg}{0.145} &  \colorbox{lg}{0.062} &  \colorbox{lg}{0.087} \\ 
  \multicolumn{10}{c}{$D=0.1$} \\
  TRUE  & 0.373 & -0.500 & 0.220 & 0.373 & 0.500 & 0.220 & 0.373 & 0.900 & 0.111 \\ 
  Bias($\thehat$) & -0.011 & 0.032 & -0.045 & -0.015 & -0.322 & -0.036 & -0.019 & -0.517 & 0.044 \\ 
  RMSE($\thehat$)  & 0.043 & \colorbox{lg}{0.408} & \colorbox{lg}{0.098} & 0.047 & 0.657 & \colorbox{lg}{0.102} & 0.066 & 0.801 & 0.099 \\
  Bias($\thehatcv$) &  -0.002 & 0.056 & -0.044 & -0.003 & -0.120 & -0.038 & -0.004 & -0.310 & 0.031 \\ 
  RMSE($\thehatcv$)  & \colorbox{lg}{0.042} & 0.482 & 0.112 & \colorbox{lg}{0.045} & \colorbox{lg}{0.504} & 0.108 & \colorbox{lg}{0.062} &  \colorbox{lg}{0.555} & \colorbox{lg}{0.090} \\ 
   \hline
\end{tabular}
\end{small}
\caption{{Comparison of the simulation based estimator $\thehat$ of \eqref{eq:defThe} and the control variates based estimator $\thehatcv$ of \eqref{eq:thecv} with $k=1$ for sample size $n = 400$. For all estimators we have taken $H=3\,000$, $p = 3$ with $w$ the Laplace density as in \eqref{eq:wtLap}. Reported results are based on $500$ replications. The models are classified by the index of dispersion $D=e^{\beta + \alpha_1}$. For each setting, the smallest RMSEs are shaded.}}\label{tb:pam}
\end{table}

%
%
{\bf Data sharing:}  Data sharing is not applicable to this article as no datasets were analyzed or used in this study.
\subsubsection*{Acknowledgement}
Thiago do R\^ego Sousa gratefully acknowledges support from the National Council for Scientific and Technological Development (CNPq - Brazil) and the TUM Graduate School. He also thanks the Statistics Department at Columbia University for its hospitality during his visit and takes pleasure to thank Viet Son Pham and Thibaut Vatter for helpful discussions.  Davis' research was partially supported by NSF grant DMS 2015379 to Columbia University.

%
%


\bibliographystyle{abbrvnat} 
\bibliography{references}       
       
\end{document}